\documentclass[11pt]{article}

\usepackage{amssymb}
\usepackage{graphicx}
\usepackage{amsmath}
\usepackage{verbatim}
\usepackage{setspace}
\usepackage{textpos}
\usepackage{changepage}
\usepackage{url}
\usepackage{fancyhdr}

\usepackage[title]{appendix}

\usepackage[dvipdfmx]{hyperref}
\usepackage{color}

\usepackage{amsthm}

\usepackage[round]{natbib}

\newtheorem{assumption}{Assumption}

\newtheorem{proposition}{Proposition}

\newtheorem{lemma}{Lemma}

\def\hat{\widehat}

\newcommand{\ve}{\varepsilon}
\newcommand{\mc}{\mathcal}
\newcommand{\mbb}{\mathbb}

\DeclareMathOperator*{\argmax}{\arg\!\max}

\begin{document}

{
\singlespacing
\title{\textbf{Asymptotic Properties of the Maximum Smoothed Partial Likelihood Estimator in the Change-Plane Cox Model}
\author{Shota Takeishi\footnote{I am grateful to my advisor Katsumi Shimotsu for his guidance and helpful suggestions. 
I would also like to thank the associate editor and the anonymous reviewer for their constructive comments.
This research was partially supported by JSPS KAKENHI Grant Number JP22J12024.}\\
Graduate School of Economics \\
University of Tokyo \\
shotakeishi2@gmail.com
}}
\maketitle
}
\begin{abstract}
	The change-plane Cox model is a popular tool for the subgroup analysis of survival data. Despite the rich literature on this model, there has been limited investigation into the asymptotic properties of the estimators of the finite-dimensional parameter. 
	Particularly, the convergence rate, not to mention the asymptotic distribution, has not been fully characterized for the general model where classification is based on multiple covariates. 
	To bridge this theoretical gap, this study proposes a maximum smoothed partial likelihood estimator and establishes the following asymptotic properties. 
	First, it shows that the convergence rate for the classification parameter can be arbitrarily close to $n^{-1}$ up to a logarithmic factor under a certain condition on covariates and the choice of tuning parameter. 
	Given this convergence rate result, it also establishes the asymptotic normality for the regression parameter.
\end{abstract}

\section{Introduction}
A change-plane model is a useful statistical framework for capturing the grouped heterogeneity of covariates effect on the outcome of interest. This method can be considered a combination of regression and classification in that it simultaneously detects a subgroup with differentiated covariates effect and infers the magnitude of this effect. Specifically, a threshold plane characterized by a single index of covariates and parameters splits the population into two groups with different regression functions. To achieve such flexibility, the literature has developed several versions of the change-plane model. 
For example, refer to \cite{seo2007smoothed}, \cite{Yu2020threshold}, \cite{lee2021factor}, and \cite{li2021multithreshold} for linear regression; \cite{mukherjee2020asymptotic} for binary classification; and \cite{zhang2021single} for quantile regression.

This study investigates the theoretical properties of an estimator for the change-plane Cox model, which is a change-plane model designed for survival analysis. Let $(T^{\circ}, Z, U, V, X) \subset \mbb R \times \mbb R^{p_1} \times \mbb R^{p_2} \times \mbb R \times \mbb R^{q}$ be a random vector defined on some underlying probability space $ (\Omega, \mc F, \mbb P)$ with the following conditional hazard model:
\begin{equation}  \label{model}
	\lambda(T^{\circ} | Z, U, V, X) = \lambda (T^{\circ}) \exp\{Z' \beta + U' \gamma 1\{V + X'\psi \geq 0 \} \},
\end{equation}
where we normalize a coefficient for $V$ to be 1 for identification purpose.
%where $\lambda_0 (t)$ is an unknown baseline hazard function of $T^{\circ}$ that does not depend on values of covariates. 
%Like a standard cox model, the estimation is based on a censored observation, the detail of which is discussed in the next section. 
In this model, $(\beta, \gamma)$ is a regression parameter where $\beta$ represents a baseline covariate effect on survival time common to the entire population, and $\gamma$ denotes an enhanced effect specific to a particular subgroup defined by a set of covariates. Conversely, $\psi$ is a classification parameter expressing how particular covariates contribute to the classification of the subgroup. 
%A higher value of a particular element of $\psi$ indicates that a subject with a high value of the corresponding covariate is more likely to belong to the subgroup.

For the subgroup analysis of the survival data, several Cox-model-based approaches have been proposed to cope with the heterogeneity of covariate effects. \cite{negassa2005tree} develop a tree-structured model with the coefficients being different across the nodes of the tree.
\cite{wu2016subgroup} propose a logistic-Cox mixture model where the model is a mixture of two Cox models with different regression coefficients and the mixing proportion is a logistic function of covariates.
Meanwhile, \cite{hu2021subgroup} introduce a Cox mixture model, in which the population is split into several latent groups and
members in each group share the same regression coefficients.

In a more direct connection to the present study, several studies have provided theoretical guarantees for the inference of the change-plane Cox model. However, these studies have not fully established the asymptotic properties of the estimators of the model parameter. 
\cite{pons2003estimation} considers a special case of the change-plane Cox model where the classification is based on only one covariate. They propose the maximum partial likelihood estimator (MPLE) and derive its asymptotic distribution. However, its theoretical result cannot be readily applied to our general model that allows for multiple classification covariates. 
\cite{wei2018change} propose the MPLE for a general case like ours, and establish its consistency; however, they do not derive the rate of convergence or asymptotic distribution of the MPLE. \cite{he2018single} deal with a model similar to the change-plane Cox model and claim that there exists a sequence of local maximizers of their objective function possessing desirable asymptotic properties. However, out of the possibly many candidates, they do not indicate which sequence possesses such properties. Particularly, it is unknown whether any sequence of the global maximizers serves as such a sequence, which often forms the main theoretical interest of the literature. 
Due to this gap, the asymptotic theory for the general change-plane Cox model is still underdeveloped. 
With regard to testing the existence of change-plane for the change-plane Cox model, \cite{kang2017subgroup} propose a sup-score test for testing the null hypothesis of no change-plane: $\gamma = 0$.

This study proposes a new estimator for the change-plane Cox model and establishes the asymptotic properties of this estimator. 
In the spirit of a smoothed estimator for a binary choice model proposed in \cite{horowitz1992smoothed}, our estimator is defined as a maximizer of the smoothed partial likelihood instead of its unsmoothed version, as in \cite{wei2018change}.
This smoothing technique sidesteps the technical difficulty associated with the non-differentiability of the indicator function, and thus enables the standard argument based on Taylor's theorem in deriving the following asymptotic results. 
First, we prove that the estimator is consistent (Proposition \ref{consistency}).
Given this result, we derive the convergence rate of the estimator of the classification parameter (Proposition \ref{roc psi}) and subsequently prove the asymptotic normality of the estimator of the regression parameter (Proposition \ref{asy msple xi}). The convergence rate of the classification parameter estimate obtained in this study can be arbitrarily close to $n^{-1}$ up to a logarithmic factor, depending on a choice of the smoothing parameter. 

This study makes an original contribution to the literature in several ways.
First, as reviewed in the above paragraph, few existing studies establish the asymptotic distribution for the regression parameter in the change-plane Cox model. 
Therefore, this study contributes to filling this gap and paves the way for statistical inferences, such as confidence interval construction, with a strong theoretical foundation and practical implications.
Second, we explicitly derive the model identifiability (Proposition \ref{identification}), which \cite{wei2018change} implicitly assume, under certain regularity conditions.
Lastly, our proof strategy for deriving the convergence rate of the classification parameter (Proposition \ref{roc psi}) is novel.
Utilizing a maximal inequality of \cite{gine2001consistency}, our close-to-$n^{-1}$ rate of the classification parameter is faster than most of the rates of the smoothed estimators in the change-plane model literature that are shown to be at most $n^{-3/4}$ \citep[e.g.,][]{seo2007smoothed, li2021multithreshold, zhang2021single}.
Furthermore, although \cite{mukherjee2020asymptotic} derives the same close-to-$n^{-1}$ rate as we do in their model, our proof is original in that we directly deal with the score function when deriving the convergence rate.
Without invoking the general rate theorem for M-estimator (e.g. Theorem 3.4.1 of \cite{van1996weak}) as \cite{mukherjee2020asymptotic} do, our direct approach shortens the proof. 
%Although there are numerous works that characterize the asymptotic properties of the smoothed estimators in the change-plane models, they only cover relatively simple cases such as linear regression and quantile regression \citep[e.g.,][]{seo2007smoothed, li2021multithreshold, zhang2021single}.
%Due to a complicated form of the partial likelihood, an extension to Cox model is nontrivial.
%From a more theoretical perspective, the literature on change-place models has started focusing on an irregular rate of convergence for the classification parameter.
%With smoothing the objective functions, the convergence rates for the classification parameter are usually shown to be slower than $n^{-3/4}$ in the literature \citep[e.g.,][]{seo2007smoothed, li2021multithreshold, zhang2021single}.
%In contrast, we derive the faster rate, which is close to $n^{-1}$, 

We note that the recent work by \cite{deng2022maximum} independently proposes an unsmoothed maximum partial likelihood estimator in the change-plane Cox model 
and establishes the convergence rate and the asymptotic distribution. While their approach is free from the choice of tuning parameter and realizes a faster $n^{-1}$ convergence rate of the classification parameter,
our smoothing approach makes attractive computational methods, such as a modified Newton-Raphson algorithm, available for the estimation as pointed out by \cite{li2021multithreshold}.
Furthermore, the proof strategies for the asymptotic results of the unsmoothed estimator and the smoothed one are distinct.
In these respects, \cite{deng2022maximum} and the present study are mutually complementary.

The rest of the paper is organized as follows. Section 2 introduces the model setup, the estimator, and the notation. Section 3 introduces our assumptions and shows the consistency of our estimator. Based on this consistency result, Section 4 introduces an additional set of assumptions and derives the rate of convergence and the asymptotic normality for the classification and regression parameters, respectively. Section 5 investigates the finite sample performance of our estimator through Monte Carlo simulations.
Subsequently, in Section 6, we analyze real-world data from AIDS Clinical Trial Group Study 175 with the proposed method. Section 7 concludes the study. Appendix A.1 provides all the proofs of the propositions in the main text, and Appendix A.2 collects the auxiliary results and their proofs.

\section{Model and Estimator}
%In this paper, we consider a change-plane Cox model for a random vector $(T^{\circ}, Z, U, V) \subset \mbb R \times \mbb R^{p_1} \times \mbb R^{p_2} \times \mbb R^{q + 1}$ defined on some underlying probability space $ (\Omega, \mc F, \mbb P)$:
%\begin{equation} \label{model}
%\lambda(T^{\circ} | Z, U, V) = \lambda (T^{\circ}) \exp\{Z' \beta + U' \gamma 1\{V'\delta \geq 0 \} \}
%\end{equation}
%where $\lambda_0 (t)$ is an unknown baseline hazard function of $T^{\circ}$ that does not depend on values of covariates. 

This study concerns the statistical inference for the finite-dimensional parameter $\theta := (\beta, \gamma, \psi)$ when the distribution of the random vector $(T^{\circ}, Z, U, V,  X)$ satisfies the conditional hazard restriction (\ref{model}) with $\theta = \theta_0$ and $\lambda(\cdot) = \lambda_0(\cdot)$. As is usual with survival analysis, due to censoring, we cannot obtain a full sample of $i.i.d.$ copies $\{ (T^{\circ}_i, Z_i, U_i, V_i, X_i) \}_{i = 1}^{n}$, where $n$ denotes the sample size. Namely, letting $C$ be a censoring time, which is also a random variable on $ (\Omega, \mc F, \mbb P)$, we can observe only a censored survival time and a censoring indicator, $(T, \delta) := (\min \{ T^{\circ}, C \}, 1\{ C \geq T^{\circ} \} )$, instead of $T^{\circ}$ itself. Hence, we conduct the inference based on $i.i.d.$ observations $\{ (T_i, \delta_i, Z_i, U_i, V_i, X_i) \}_{i = 1}^{n}$.

For an estimation of $\theta$, we define a partial likelihood as
\begin{equation} \label{partial likelihood}
	L_n (\theta) := \prod_{i \in \{k : \delta_k = 1 \}} \frac{\exp\{\eta (W_i, \theta) \}}{n^{-1} \sum_{j \in \{k : T_k \geq T_i \}} \exp \{\eta(W_j, \theta) \}},
\end{equation}
where we collect covariates $W = (Z, U, V, X)$ and define $\eta (W, \theta) := Z'\beta + U'\gamma 1 \{ V + X'\psi \geq 0 \}$ for brevity. Similarly, we define its scaled logarithmic version as $l_n (\theta) := n^{-1} \log L_n (\theta)$. Instead of directly maximizing $l_n (\theta)$, we maximize a smoothed version of this log-likelihood to estimate $\theta$. 
Specifically, we replace $\eta (W, \theta)$ with its smooth counterpart $\eta_n (W, \theta) = Z'\beta + U'\gamma \mathcal K \left( (V + X'\psi)/h_n \right)$ where $\mathcal K (\cdot)$ and $h_n$ are a smooth kernel and a smoothing parameter, whose detail is discussed in the following section. 
We now define a smoothed log partial likelihood (SLPL) as
\begin{equation} \label{smoothed log likelihood}
	l^*_{n} (\theta) := \frac{1}{n} \sum_{i = 1}^{n} \delta_i \left[\eta_n (W_i, \theta)
	- \log \left( \frac{1}{n} \sum_{j = 1}^{n} Y_j (T_i) \exp \left\{ \eta_n (W_j, \theta) \right\} \right) \right],
\end{equation}
where $Y_i (t) := 1\{ T_i \geq t \}$ is an at-risk process for $t \in [0, \infty)$. Then, we define a maximum smoothed partial likelihood estimator (MSPLE) $\hat \theta_n$ for $\theta$ as
\begin{equation} \label{MSPLE}
	\hat{\theta}_n := \argmax_{\theta \in \Theta} l^*_{n} (\theta),
\end{equation}
where $\Theta$ is the parameter space for $\theta$, which is discussed in the following section.

In the remainder of the paper, we use the following notations. Let $:=$ denote ``equals by definition.'' For $a \in \mbb R^{k}$, $\| a \|$ denotes a Euclidean norm of $a$. Define a matrix $a^{\otimes 2} := a a'$ for a real-valued vector $a$. For a $l \times m$ real-valued matrix $A = (a_{ij})_{ij}$, define $|A| := \max_{1 \leq i \leq l, 1 \leq j \leq m} |a_{i, j}|$. Let $\mc C$ and $\mc D$ be universal finite positive constants whose value may change from one expression to another. ``$A$ almost surely'' or equivalently ``$A \ a.s.$'' means that there exists $F \in \mc F$ with $\mbb P (F) = 0$ such that $\{ \omega \in \Omega : A \ \text{does not hold} \} \subset F$. For a sequence of random vectors $\{X_n \}_{n \in \mbb N}$ on $(\Omega, \mc F, \mbb P)$, let $X_n = o_p (1)$ denote ``$\| X_n \|$ converges to zero in probability as $n \rightarrow \infty$.'' 
For random vectors $Y_1$ and $Y_2$, let $F_{Y_2} (\cdot)$, $F_{Y_1 | Y_2} (\cdot | \cdot)$, and $f_{Y_1 | Y_2} (\cdot | \cdot)$ be a distribution of $Y_2$; a conditional distribution of $Y_1$, given $Y_2$; and a conditional density of $Y_1$, given $Y_2$ with respect to the Lebesgue measure; respectively. For notational convenience, we abbreviate 
\begin{align}
	\Lambda (t) &:= \int_{0}^{t} \lambda (u) du \label{Lambda}.
\end{align}
Particularly, let $\Lambda_0 (t)$ denote $\int_{0}^t \lambda_0 (u) du$. All the limits below are taken as $n \rightarrow \infty$, unless stated otherwise.

\section{Consistency of the MSPLE}
This section proves the consistency of the MSPLE defined in the previous section. 
We first introduce the following assumptions. \cite{wei2018change} prove the consistency of the non-smoothed MPLE with continuously distributed covariates $X$ under similar assumptions to ours.
Compared with \cite{wei2018change}, we additionally assume Assumption \ref{the parameter space}(b) and Assumptions \ref{covariates assumption}(b)-(d) in order to rigorously establish the model identifiability and relax the continuity assumption (Condition 3 of \cite{wei2018change}) to allow for discrete covariates in $X$.
%\begin{assumption}%[The basic setup]
%\label{the basic setup}
%(a) The random vector $(T^{\circ}, Z, U, X, \tilde X)$ satisfies the conditional hazard restriction (\ref{model}) with $\theta = \theta_0$ and $\lambda = \lambda_0$. (b) $\{(T^{\circ}_i, C_i, Z_i, U_i, X_i, \tilde X_i) \}_{i = 1}^n$ are $i.i.d.$ copies of a random vector $(T, C, Z, U, X, \tilde X)$
%\end{assumption}

\begin{assumption}%[A survival time and a censoring time]
\label{a survival time and a censoring time}
(a) The distribution of $T^{\circ}$ has a continuous conditional density with respect to the Lebesgue measure, given covariates $W$. (b) There exists $\tau \in (0, \infty)$ such that $\mbb P(C \geq \tau | W) = \mbb P(C = \tau | W) > 0$ a.s. and $\mbb P(T^{\circ} \geq \tau | W) > 0$ a.s. (c) $T^{\circ}$ and $C$ are statistically independent given covariates $W$.
\end{assumption}

\begin{assumption}%[The parameter space]
\label{the parameter space}
(a) The parameter space $\Theta = \Theta^{\beta} \times \Theta^{\gamma} \times \Theta^{\psi}$, for $\theta$, is a compact convex subset of $\mbb R^{p_1 + p_2 + q}$, and the true parameter $\theta_0$ lies in the interior $\Theta^{\circ}$ of $\Theta$. 
(b) The parameter space $\Pi$, for $\lambda$, is a collection of non-negative and nonzero continuous functions. (c) The true parameter $\gamma_0$ satisfies $\gamma_0 \neq 0$
\end{assumption}

\begin{assumption}%[Covariates]
\label{covariates assumption}
(a) $(Z, U, X)$ lies in a bounded subset of $\mbb R^{p_1 + p_2 + q }$. (b) Conditionally on $X$, the distribution of $V$ has a positive density everywhere with respect to the Lebesgue measure a.s.
%(c) For any $\ve > 0$, $P\left(Z \in U_{\ve} (0), X < \inf \{ \tilde x ' \psi : \tilde x \in \text{supp} (\tilde X), \theta \in \Theta \} \right) > 0$ where $U_{\ve}(z)$ is an $\ve$ neighborhood of a point $z$.
(c) $\mathbb E[X X']$ is positive definite.
(d) $Var(Z| V, X)$ and $Var(U | V, X)$ are positive definite a.s.
\end{assumption}

The continuity of the survival time $T^{\circ}$ in Assumption \ref{a survival time and a censoring time}(a) is often assumed in the Cox model (see page 425 on \cite{van2000asymptotic}, for example). Assumption \ref{a survival time and a censoring time}(b) is common in survival analysis. Practically, this assumption is satisfied if the follow-up of subjects is terminated all at once, at some point in the study, with some portion of the population still surviving and uncensored. Assumption \ref{a survival time and a censoring time}(c) is called independent censoring in the literature and is assumed on page 443 on \cite{pons2003estimation} and in Assumption 2 of \cite{wei2018change}. %The compactness in Assumption \ref{the parameter space} (a) is standard in an M-estimation problem. 
Continuity in Assumption \ref{the parameter space}(b) is mainly for the model identifiability. This condition covers a variety of popular distributions of $T^{\circ}$ in survival analysis such as Weibull distribution. Assumption \ref{the parameter space}(c) is required for identifiability of $\psi_0$. In other words, there has to be a subgroup. Assumption \ref{covariates assumption}(a) is set mainly for the relevant class of functions to belong to a Glivenko-Cantelli class in the proof of the consistency. 
Assumption \ref{covariates assumption}(b) is for identifiability and is usually assumed in the change-plane models. For example, see Assumption 1(d) of \cite{seo2007smoothed} and Assumption 1 of \cite{mukherjee2020asymptotic}. This assumption can be weakened to allow for bounded $V$ as ``Conditionally on $X$, the distribution of $V$ has a positive density on an interval $A$ such that $\inf A < \inf_{\psi \in \Theta^{\psi}, x \in \text{supp}(X) } - |x'\psi| $ and $\sup A > \sup_{\psi \in \Theta^{\psi}, x \in \text{supp}(X) } |x'\psi| $," where $\text{supp} (X)$ is a support of the distribution of $X$. 
Intuitively, the support of $V$ has to cover a possible variation in $X'\psi$. For simplicity, however, we impose the unbounded condition throughout.  
%It must be noted that Assumptions \ref{covariates assumption}(a) and (b) are weaker than assumptions on covariates in \cite{wei2018change} in the sense that we allow $X$ to contain discrete variables. Assumption \ref{covariates assumption}(d) rules out cases where either $Z$ or $U$ is a subset of $(V, X)$.

%\begin{assumption}[The moment condition]
%\label{moment condition}
%(a) $Var(U| X, \tilde X)$ is a positive definite matrix $\pas$(b) $Var(Z | X)$ is a positive definite matrix P-a.s. (c) The density of $(\delta, T, Z, U, X, \tilde X)$ given $(\theta, \lambda)$, 
%$f(\delta, T, Z, U, X, \tilde X | \theta, \lambda)$ satisfies $E[|\log f(\delta, T, Z, U, X, \tilde X | \theta, \lambda)|] < \infty$ for all $\theta \in \Theta$ and $\lambda \in \Pi$.
%\end{assumption}

The following proposition shows the identifiability of $(\theta_0, \lambda_0)$, which is a prerequisite for consistency.

\begin{proposition}\label{Identification}
\label{identification}
Assume Assumptions \ref{a survival time and a censoring time}-\ref{covariates assumption} hold. Then, $(\theta_0, \lambda_0)$ is identified from the distribution of $(T, \delta, W)$.
\end{proposition}

As reviewed in the introduction, the consistency result in \cite{wei2018change} also hinges on model identifiability, but they implicitly assume identifiability (see page 902 on \cite{wei2018change}).

We make an additional assumption concerning the properties of a kernel function $\mc K$ and a sequence of tuning parameters $h_n$ that appear in (\ref{smoothed log likelihood}).

\begin{assumption}\label{kernel and smoothing}
(a) A kernel function $\mc K: \mbb R \rightarrow \mbb R$ is a monotone continuous function such that $\lim_{s \rightarrow - \infty} \mc K (s) = 0$ and $\lim_{s \rightarrow \infty} \mc K (s) = 1$. (b) $\{ h_n \}_{n = 1}^{\infty}$ is a sequence of positive real numbers that satisfies $\lim_{n \rightarrow \infty} h_n = 0$.
\end{assumption}

The basic idea of Assumption \ref{kernel and smoothing} is to approximate the indicator function by the c.d.f. of a random variable, with the precision of the approximation increasing with an increase in the sample size. This idea appears in many estimation problems where originally discontinuous objective functions are replaced by their smooth counterparts \citep[e.g.,][]{horowitz1992smoothed, seo2007smoothed,zhang2021single}.

The following proposition shows consistency of $\widehat \theta_n$.

\begin{proposition}\label{consistency}
Assume Assumptions \ref{a survival time and a censoring time}-\ref{kernel and smoothing} hold. Then, $\hat{\theta}_n$ converges to $\theta_0$ almost surely.
\end{proposition}

\section{Asymptotic Distribution of MSPLE}
This section derives the asymptotic distribution of the MSPLE for $\xi := (\beta, \gamma)$ and establishes the convergence rate of the MSPLE for $\psi$. When $\hat{\theta}_n$ lies in a sufficiently small neighborhood of $\theta_0$,  the first-order condition and applying Taylor's Theorem coordinate-wise give %When $\hat{\theta}_n$ lies in an interior $\Theta^{\circ}$, the first order condition and Taylor's Theorem gives
\begin{align}\label{expansion}
	0 = n^{1/2} \nabla_{\theta}  l^*_{n} (\hat{\theta}_n) =  n^{1/2} \nabla_{\theta}  l^*_{n} (\theta_0) +  \nabla_{\theta \theta'} l^*_{n} (\Bar \theta_n) n^{1/2} (\hat{\theta}_n - \theta_0).
\end{align}
where $\Bar \theta_n$ lies on the path connecting $\widehat \theta_n$ and $\theta_0$ and may take different values row by row of the matrix.
%where $\bar{\theta}_n$ lies between $\hat{\theta}_n$ and $\theta_0$ and may take different value across different rows of $\nabla_{\theta \theta'} l^*_{n} (\bar{\theta}_n)$. 
Particularly, taking the first $p_1 + p_2$ elements of (\ref{expansion}) yields
\begin{align}\label{expansion xi 0}
	\nabla_{\xi \xi'} l^*_{n} (\Bar \theta_n) n^{1/2} (\hat{\xi}_n - \xi_0)  = - n^{1/2} \nabla_{\xi}  l^*_{n} (\theta_0) - \nabla_{\xi \psi'} l^*_{n} (\Bar \theta_n) n^{1/2} (\hat{\psi}_n - \psi_0).
\end{align}
In the following, we show $\hat{\psi}_n - \psi_0 = O_p (h_n)$ so that 
$\nabla_{\xi \psi'} l^*_{n} (\Bar \theta_n)  n^{1/2} (\hat{\psi}_n - \psi_0) = o_p (1)$, under an additional assumption on the rate of convergence of $h_n$. 
Then, we apply the central limit theorem to $n^{1/2} \nabla_{\xi}  l^*_{n} (\theta_0)$ and 
show a convergence in the probability of $\nabla_{\xi \xi'} l^*_{n} (\Bar \theta_n)$ to some positive definite matrix in order to derive the asymptotic distribution of $n^{1/2}(\hat{\xi}_n - \xi_0)$.

We introduce an additional set of assumptions.
These types of assumptions are common in the estimation problem where an indicator function in the model is smoothed by a kernel function \citep[e.g.,][]{horowitz1992smoothed, seo2007smoothed, li2021multithreshold}.
We, however, relax the assumption on the rate of $h_n$ (Assumption \ref{kernel and smoothing 2}(c)), compared to most of the works in the literature.

\begin{assumption}\label{kernel and smoothing 2}
(a) A kernel function $\mc K$ is differentiable with its first derivative being of bounded 1-variation: $\sup \left\{\sum_{i = 1}^n |\mathcal K' (x_i) - \mathcal K' (x_{i - 1})|: -\infty < x_0 < \cdots < x_n < \infty, n \in \mbb N \right\} < \infty$ (b) The first derivative of a kernel function $\mc K$ satisfies $\int \mc |s|^{i} | \mc K'(s)|^{j} ds < \infty$ for $i = 0, 1$ and $j = 1, 2$ and  $\lim_{s \rightarrow \pm \infty} s^2 \mc K'(s) = 0$. (c) $\{ h_n \}_{n \in \mbb N}$ satisfies $(\log n)/(n h_n) \rightarrow 0$ and $n^{1/2} h_n \rightarrow 0$.
\end{assumption}

\begin{assumption}\label{density}
A random variable $S := V + X'\psi_0$ has a differentiable conditional density $f_{S | (T, Z, U, X)} (\cdot | r)$ for a given $(T, Z, U, X) = r$ with respect to the Lebesgue measure with the following conditions: \\ $\sup_{(s, r) \in \mbb R^2} f_{S | (T, Z, U, X)} (s | r) < \infty$ and $\sup_{(s, r) \in \mbb R^2} |f_{S | (T, Z, U, X)}' (s | r)| < \infty$.
\end{assumption}

\begin{assumption}\label{moment condition 2}
(a) $\mathbb E[Y(\tau) X X' (U' \gamma_0)^2 f_{S | (T, Z, U, X)} (0 | T, Z, U, X)]$ is positive definite. \\ (b) $\mathbb E[Y(\tau) e^{Z'\beta_0} (U' \gamma_0 - e^{U'\gamma_0} + 1)f_{S|T, Z, U, X} (0 | T, Z, U, X) | X] < 0 \ a.s.$ \\ and $\mathbb E[Y(\tau)e^{Z'\beta_0} (e^{U'\gamma_0} (U'\gamma_0 - 1) + 1) f_{S|T, Z, U, X} (0 | T, Z, U, X) | X] > 0 \ a.s.$
\end{assumption}

Assumptions \ref{kernel and smoothing 2}(a) and (b) are concerned with the properties of a kernel function $\mc K$, and they are similar to those in \cite{horowitz1992smoothed} and \cite{seo2007smoothed}. A sufficient condition for Assumption \ref{kernel and smoothing 2}(a) is $\int |\mc K'' (s)| ds < \infty$, which is implied by  Assumption 7(a) of \cite{horowitz1992smoothed} or Assumption 3 of \cite{seo2007smoothed}. The c.d.f. of a Gaussian distribution, for example, qualifies as $\mc K$ satisfying Assumptions \ref{kernel and smoothing 2}(a) and (b). Assumption \ref{kernel and smoothing 2}(c) indicates that the speed at which $h_n$ converges to zero can be arbitrarily close to the order of $n^{-1}$ up to a logarithmic factor. This convergence 
rate of $h_n$ is fast, relative to the one appearing in the literature of a change-plane linear regression model, where the rate cannot be faster than the order of $n^{- 1/2}$. For example, see Assumption 3(e) in \cite{seo2007smoothed} and Condition 6 in \cite{li2021multithreshold}. However, for a binary classification model, similar to this study, \cite{mukherjee2020asymptotic} assume $(\log n)/(n h_n) \rightarrow 0$. Assumption \ref{density} is similar to Assumption 2(e) of \cite{seo2007smoothed}, though our assumption is more restrictive in assuming the boundedness of the first derivative. 
Assumption \ref{moment condition 2}(b) is not strong because $s - e^s + 1 \leq 0$ and $e^s (s - 1) + 1 \geq 0$ for any $s \in \mbb R$.

Under these assumptions, we obtain the convergence rate of the MSPLE $\hat \psi_n$ for $\psi$.

\begin{proposition}\label{roc psi}
	Assume Assumptions \ref{a survival time and a censoring time}-\ref{moment condition 2} hold. Then, $\hat \psi_n - \psi_0 = O_p (h_n)$.
\end{proposition}

In view of Assumption \ref{kernel and smoothing 2}(c), the convergence rate of $\hat \psi_n$ can be arbitrarily close to $n^{-1}$ up to a logarithmic factor. Although our basic proof strategy follows that of \cite{seo2007smoothed}, their rate is not faster than $n^{-3/4}$. The superiority of our rate stems from the utilization of a maximal inequality (inequality (2.5), in \cite{gine2001consistency}), in approximating the sample score function by its population counterpart. Refer to the proof of Lemma \ref{lem_roc} for detail. We note that \cite{mukherjee2020asymptotic} use the same maximal inequality and derive the same rate for a different model, though their proof strategy is distinct from that of ours.

Subsequently, we show the asymptotic normality of the score function \\ $n^{1/2} \nabla_{\xi}  l^*_{n} (\theta_0)$ and the convergence in the probability of a matrix $\int_{0}^1 \nabla_{\xi \xi'} l^*_{n} (\theta_0 + t(\hat{\theta}_n - \theta_0)) dt$ in (\ref{expansion xi}). 
For preparation, define $\phi (\theta)_i := (Z_i', U_i' 1\{ V_i + X_i' \psi \geq 0 \})'$, \\ $\phi^{(n)} (\theta)_i := (Z_i', U_i' \mc K((V_i + X_i'\psi)/h_n))'$, $\pi^{(0)} (t) := \mathbb E[Y_i(t) e^{\eta_0 (W_i, \theta_0)}]$, \\ $\pi^{(1)} (t) := \mathbb E[Y_i(t) \phi (\theta_0) e^{\eta_0 (W_i, \theta_0)}]$, and $\pi^{(2)} (t) := \mathbb E[Y_i(t) \phi(\theta_0)_i \phi(\theta_0)_i' e^{\eta_0 (W_i, \theta_0)} ]$. Then, we set the following assumption. 
\begin{assumption}\label{pd variance matrix}
	$\mc I$ := $\int_{0}^{\tau} ( \pi^{(2)} (t) - \pi^{(1)} (t) \pi^{(1)} (t)' / \pi^{(0)} (t) ) d\Lambda_0 (t)$ is positive definite.
\end{assumption}

This assumption is standard in the Cox model. See, for example, Theorem 8.4.1 of \cite{fleming2011counting}. The proof of Proposition \ref{convergence matrices} in the following indicates that $\mc I$ is the probability limit of  $- \nabla_{\xi \xi'} l^{\ast}_n (\widetilde \theta_n)$ for any consistent estimator $\widetilde \theta_n$ of $\theta_0$. 
Since $- \nabla_{\xi \xi'} l^*_{n} (\widetilde \theta_n) $ is positive semi-definite (e.g., see page 304 in \cite{fleming2011counting}), its probability limit $\mathcal I$ is at least positive semi-definite.
In this sense, Assumption \ref{pd variance matrix} strengthens this positive semi-definiteness to positive definiteness.

The following propositions show the asymptotic normality of a scaled score function and the probability convergence of the Hessian matrix.

\begin{proposition}\label{a score}
	Assume Assumptions \ref{a survival time and a censoring time}-\ref{pd variance matrix} hold. Then,
		$n^{1/2} \nabla_{\xi}  l^*_{n} (\theta_0) \rightarrow_{d} N \left(0, \mc I \right)$.
\end{proposition}

\begin{proposition}\label{convergence matrices}
Assume Assumptions \ref{a survival time and a censoring time}-\ref{moment condition 2} hold. Then, $\nabla_{\xi \xi'} l^*_{n} (\Bar \theta_n)$ converges to $- \mathcal I$ in probability.
%\begin{align*}
%(a) & \quad  \int_{0}^1 \nabla_{\xi \xi'} l^*_{n} (\theta_0 + t(\hat{\theta}_n - \theta_0)) dt \rightarrow - \mc I \ \pas \\
%(b) & \quad  | \nabla_{\xi \psi'} l^*_{n} (\bar{\theta}_n)|
%\end{align*}
\end{proposition}

Combining these propositions gives the asymptotic normality of the MSPLE for the regression parameter $\xi$.
\begin{proposition}\label{asy msple xi}
Assume Assumptions \ref{a survival time and a censoring time}-\ref{pd variance matrix} hold. Then,
$n^{1/2} (\hat \xi_n - \xi_0) \rightarrow_d N \left(0, \mc I^{-1} \right)$.
\end{proposition}

It must be noted that our estimator $\hat \xi_n$ is adaptive in the sense that its asymptotic distribution is the same as the case when $\psi_0$ is known and the model reduces to the usual Cox model. In other words, the estimation of $\xi$ is not affected by that of $\psi$ asymptotically. For the asymptotic distribution for the Cox model, see, for example, Chapter 8 of \cite{fleming2011counting}. This adaptability also appears in the asymptotic distribution for the change-point Cox model considered in \cite{pons2003estimation}. In this respect, Proposition \ref{asy msple xi} can be viewed as a multivariate extension of Theorem 5 in \cite{pons2003estimation} to our general model.

\section{Monte Carlo Simulation}

In this section, we conduct a simulation study to assess the finite sample performance of the MSPLE. We use the Julia language for the simulation, with a random seed set to 20220209 for each sample size and each parameter value specified below.

We generate 1,000 replications of $i.i.d.$ random variables $\{ T_i, C_i, Z_i, V_i, X_i \}_{i = 1}^n$ of sample sizes $n = 200, 500$, and $1000$ from the following change-plane Cox model:
\[
	\lambda(T^{\circ} | Z, V, X) = \lambda \exp\{Z \beta + Z \gamma 1\{V + \psi_1 + X \psi_2 \geq 0 \} \},
\]
with $\lambda = 0.1$, $\beta = 0.8$, $\psi_1 = 0.4$, and $\psi_2 = 0.3$. For the distribution of covariates, we set $Z \sim Bernoulli(0.5)$, $V \sim N(-2, 4), X \sim U(-0.5, 0.5)$, and $C = 15$, and they are mutually independent. 
Following \cite{seo2005working}, we try different values of $\gamma \in \{ 0.25, 0.5, 1 \}$ to evaluate the effect of the magnitude of this value on the estimation.
To calculate the MSPLE, we must maximize $l^{\ast}_n (\theta)$ in (\ref{smoothed log likelihood}). We set $\mc K(\cdot)$ to the c.d.f. of the standard normal distribution and $h_n = (\log n)^2 / n$. We also adopt the following heuristic algorithm for the maximization:

\begin{enumerate}
\item[1.] Fix an initial value of $\hat \psi$.
\item[2.] With $\hat \psi$ fixed, maximize $l^{\ast}_n (\xi, \hat \psi)$ with respect to $\xi$ and update $\hat \xi$ by setting $\hat \xi$ to the maximizer.
\item[3.] With $\hat \xi$ fixed, maximize $l^{\ast}_n (\hat \xi, \psi)$ with respect to $\psi$ and update $\hat \psi$ by setting $\hat \psi$ to the maximizer.
\item[4.] Iterate steps 2 and 3 until the convergence.
\end{enumerate}
A similar iterative algorithm is used in \cite{li2021multithreshold}.
We implement this algorithm multiple times with initial values $(\widehat \psi_1, \widehat \psi_2)$ in $[-1, 1] \times [-1, 1]$, in step 1; out of the candidates, we adopt $\hat \theta$ that maximizes $l^{\ast}_n (\theta)$ as the MSPLE. It must be noted that, since $l^{\ast}_n (\xi, \hat \psi)$  is concave with respect to $\xi$, for fixed $\hat \psi$, we apply the standard Newton-Raphson method in step 2. 
To deal with non-concavity of $l^{\ast}_n (\hat \xi, \psi)$ with respect to $\psi$, for fixed $\hat \xi$, we use a gradient descent method for step 3.

For evaluating the resulting MSPLE of $(\beta, \gamma)$, we construct the asymptotic 95\% confidence intervals based on the asymptotic distribution obtained in Proposition \ref{asy msple xi}, with a variance matrix $\mc I^{-1}$ replaced with its estimator $- \nabla_{\xi \xi'} l^{\ast}_n (\hat \theta_n)^{-1}$. By the direct application of the proof of Proposition \ref{convergence matrices}, $- \nabla_{\xi \xi'} l^{\ast}_n (\hat \theta_n)^{-1}$ is consistent for $\mc I^{-1}$. 
However, we assess the performance of $\hat \psi_n$ by the mean classification error (MCE). We calculate the classification error as $n^{-1} \sum_{i = 1}^n | 1\{ V_i + \psi_1 + X_i \psi_2 \geq 0 \} - 1\{ V_i + \widehat \psi_1 + X_i \widehat \psi_2 \geq 0\}|$, for each replication. This index evaluates the in-sample accuracy of subgroup prediction.

\begin{table}[h]\caption{Coverage probability and MCE} \label{table 1} \medskip 
	\centering 
	\begin{tabular}{c|ccc|ccc|ccc}
		\hline
		\multicolumn{1}{c|}{} & \multicolumn{3}{c|}{$\gamma = 0.25$} & \multicolumn{3}{c|}{$\gamma = 0.5$} & \multicolumn{3}{c}{$\gamma = 1.0$} \\
		\hline
		& $\beta$ & $\gamma$  & MCE & $\beta$ & $\gamma$  & MCE & $\beta$ & $\gamma$ & MCE \\
		\hline
		$n = 200$ & 0.909 & 0.709 & 0.121 & 0.933 & 0.859 & 0.097 & 0.938 & 0.912 & 0.055 \\
		$n = 500$ & 0.926 & 0.830 & 0.099 & 0.946 & 0.905 & 0.064 & 0.947 & 0.935 & 0.024 \\
		$n = 1000$ & 0.943 & 0.883 & 0.071 & 0.951 & 0.926 & 0.038 & 0.953 & 0.946 & 0.012 \\
		\hline
	\end{tabular} \\
	\begin{flushleft}
	\small
	Notes: The columns ``$\beta$" and ``$\gamma$" report the frequency at which the 95\% confidence interval contains the true value of respective parameters. 
	In column ``MCE", we average over all the classification errors of the replications.
	\end{flushleft}
\end{table}

Table \ref{table 1} reports the result of the simulation.
Overall, the performance improves with the sample size and with the value of $\gamma$, which is similar to the result of \cite{seo2005working}.
Still, there is a remarkable contrast between the result of coverage probability for $\beta$ and $\gamma$.
The confidence interval for $\beta$ has a coverage probability more than 0.9 even with a moderate sample size $n = 200$ for all the values of $\gamma$.
However, the confidence interval for $\gamma$ performs poorly when $\gamma$ is small, which indicates that subgroups need to be well-separated for the inference to work properly.
As to the mean classification error, the estimated model predicts the subgroup membership correctly nearly 90\% of the time even in the worst case. 

Additionally, we evaluate the impact of the choice of a kernel $\mathcal K (\cdot)$ and the tuning parameter $h_n$ for a small sample size.
The simulation setting is the same as that for Table 1 with $\gamma = 1$ and $n = 200$. We try two kernel functions, $\Phi (x):$ the c.d.f. of standard normal distribution and 
$\text{Logit} (x) := e^{x} / (1 + e^{x})$, and different values of the tuning parameter:
 \[ h_n \in \{ (\log n)^2 / n^{1/2}, (\log n)^2 / n^{2/3}, (\log n)^2 / n, (\log n)^2 / n^{4/3}, (\log n)^2 / n^{3/2}  \}. \]
We also assess the effect of smoothing itself. Namely, we estimate the model parameter with unsmoothed objective function (\ref{partial likelihood}) where the maximization of the objective function with respect to $\psi$ is conducted via grid search
with grid length $0.01$. 

\begin{table}[h]\caption{Effect of $\mathcal K$ and $h_n$} \label{table 2} \medskip 
	\centering 
	\begin{tabular}{c|ccc|ccc||ccc}
		\hline
		\multicolumn{1}{c|}{} & \multicolumn{3}{c|}{$\mathcal K = \Phi$} & \multicolumn{3}{c||}{$\mathcal K = \text{Logit}$} & \multicolumn{3}{c}{Unsmoothed} \\
		\hline
		& $\beta$ & $\gamma$  & MCE & $\beta$ & $\gamma$  & MCE & $\beta$ & $\gamma$ & MCE \\
		\hline
		$h_n = (\log n)^2 / n^{1/2}$ & 0.915 & 0.817 & 0.083 & 0.849 & 0.699 & 0.109 & 0.941 & 0.937 & 0.019 \\
		$h_n = (\log n)^2 / n^{2/3}$ & 0.938 & 0.899 & 0.061 & 0.923 & 0.866 & 0.069 & & & \\
		$h_n = (\log n)^2 / n$ & 0.938 & 0.912 & 0.055 & 0.940 & 0.915 & 0.055 & & & \\
		$h_n = (\log n)^2 / n^{4/3}$ & 0.942 & 0.918& 0.039& 0.943 & 0.923 & 0.035 & & & \\
		$h_n = (\log n)^2 / n^{3/2}$ & 0.940 & 0.925 & 0.033 & 0.080 & 0.081 & 0.035 & & & \\
		\hline
	\end{tabular} \\
	\begin{flushleft}
	\small
	Notes: We set $n = 200$ and $\gamma = 1$. The columns ``$\beta$" and ``$\gamma$" report the frequency at which the 95\% confidence interval contains the true value of respective parameters. 
	In column ``MCE", we average over all the classification errors of the replications.
	\end{flushleft}
\end{table}

As table \ref{table 2} suggests, the performance tends to improve as the value $h_n$ gets smaller. However, when $\mathcal K = \text{Logit}$ and $h_n = (\log n)^2 / n^{3/2}$, our algorithm fails to calculate the maximizer properly.
This is due to the fact that matrix operation becomes numerically unstable. Hence, it is not advised to set the value of $h_n$ too small.
With regard to the kernels, different choices of $\mathcal K$ do not affect the performance significantly except for the extreme cases with $h_n = (\log n)^2 / n^{1/2}$ and $h_n = (\log n)^2 / n^{3/2}$.
Lastly, estimation with the unsmoothed objective function moderately refines the performance of the coverage probability for $\gamma$ and MCE.
However, as indicated in the introduction section, this improvement comes at a cost in terms of computational time. In fact, in our computing environment, Apple M1 CPU (8 cores) and 16 GB RAM, it 
takes 246.797 seconds to compute an estimate with the nonsmoothed objective function whilst the time reduces to 4.857 seconds using the smoothed one with $\mathcal K = \Psi$ and $h_n = (\log n)^2 / n$.

\section{Real-World Data Analysis}
We apply our proposed method to real-world data from AIDS Clinical Trials Group Study 175 (ACTG175).
In this study, HIV-infected patients are randomized into 4 different medications: zidovudine only, zidovudine plus didanosine, zidovudine plus zalcitabine, and didanosine only.
Refer to \cite{hammer1996trial} for the detail of the experiment. The data, ``ACTG175", are available in R package \textbf{speff2trial}.
Following \cite{lu2013variable} and \cite{kang2017subgroup}, our analysis focuses on two groups: the one receiving zidovudine only and the other receiving either of the other three treatments.
We thus label the former group as the control group and the latter as the treatment group. 
A survival time is a time up to either of the following events: progressing to AIDS, having a larger than 50\% decline in the CD4 count, and death.
We include age and homosexual activity (0 = no, 1 = yes), denoted by "homo", as covariates used for subgroup classification, as in 
\cite{lu2013variable} and \cite{kang2017subgroup}. Note that we standardize the age variable so that the support of this covariate is more likely to cover a possible variation in $X'\psi$ for the purpose of identification, as discussed in an argument regarding Assumption \ref{covariates assumption}(b).
\cite{kang2017subgroup} use the same dataset to test the null hypothesis that there exists no subgroup, characterized by the two covariates above, with different treatment effects. 
While \cite{kang2017subgroup} reject the null hypothesis of no subgroup,
they do not prove the asymptotic distribution of their estimator.  Our analysis here provides the standard error for estimates of the treatment effects and, accordingly, 95\% confidence intervals for those parameters based on the asymptotic distribution derived in Proposition \ref{asy msple xi}.

\begin{table}[!htb]
	\caption{Estimates of the model $(a)$ for ACTG175}
	\label{table3} \medskip
	\centering 
	\begin{tabular}{c|ccc}
		\hline
		& Estimate & S.E. & CI \\ \hline \hline  
		$\gamma$&  -0.556 & 0.091 & [-0.734, -0.378]\\ 
		$\psi_1$& 1.477 & \\  
		$\psi_2$& -0.298 & \\ \hline 
	\end{tabular} \\
	\small
	Notes: The standard error is constructed from $- \nabla_{\gamma \gamma'} l^*_n (\widehat \theta_n)^{-1}$.
	The column ``CI'' reports 95\% confidence intervals for $\gamma$.
\end{table}

\begin{table}[!htb]
	\caption{Estimates of the model $(b)$ for ACTG175}
	\label{table4} \medskip
	\centering 
	\begin{tabular}{c|ccc}
		\hline
		& Estimate & S.E. & CI \\ \hline \hline
		$\beta$& -0.675 & 0.095 & [-0.862, -0.489]\\  
		$\gamma$&  0.342 & 0.145 & [0.056, 0.627]\\ 
		$\psi_1$& -1.930 & \\  
		$\psi_2$& 0.979 & \\ \hline 
	\end{tabular} \\
	\small
	Notes: The standard error is constructed from $- \nabla_{\xi \xi'} l^*_n (\widehat \theta_n)^{-1}$.
	The column ``CI'' reports 95\% confidence intervals for $\beta$ and $\gamma$.
\end{table}

Let $Z = \text{treatment}$, $V = \text{age}$, $X_1 = 1$ and  $X_2 = \text{homo}$. 
We consider the following two models: 
\[ (a) \ \eta (W, \theta) = Z\gamma 1 \{ V + \psi_1 + X \psi_2 \geq 0 \} \ \text{and} \ (b) \ \eta (W, \theta) = Z\beta + Z\gamma 1 \{ V + \psi_1 + X \psi_2 \geq 0 \}. \]
The model $(a)$ follows the setting of \cite{kang2017subgroup} where the model assumes away the treatment effect outside the subgroup.
In contrast, the model $(b)$ is more general in that it includes the main effect term $Z\beta$.
We perform the proposed estimation procedure with kernel $\mathcal K = \Phi$ and tuning parameter $h_n = (\log n)^2 / n$ as in Table \ref{table 1}.
Table \ref{table3} and \ref{table4} report the parameter estimates, the standard error of some of those estimates and 95\% confidence intervals for $\gamma$ and $(\beta, \gamma)$ for the model $(a)$ and $(b)$, respectively.
The estimated regression model for $(a)$ is $\eta(W, \widehat \theta_n) = -0.556 Z \times 1 \{ V - 0.298 X_2 \geq -1.477 \}$
with the estimate of $\gamma$ being significant at 5 \% level.
The result suggests that the treatment is likely to result in a longer survival time for older patients and in the absence of homosexual activity, which is consistent with the result in \cite{kang2017subgroup}. 
Meanwhile, the estimated regression model for $(b)$ is $\eta (W, \widehat \theta_n) = -0.675 Z + 0.342 Z \times 1 \{ V + 0.979 \geq 1.930 \}$ with the estimates of $\beta$ and $\gamma$ being significant at 5 \% level.
Overall, treatment lengthens survival time; however, aging and homosexual activities can weaken such effects.
Notably, the direction of the effect of aging on the treatment effect in the model $(b)$ is opposite to that in the model $(a)$ and, hence, that in \cite{kang2017subgroup}.
This difference signifies the importance of the model specification in the change-plane Cox model.

\section{Conclusion}
This study analyzes the asymptotic properties of the MSPLE for the change-plane Cox model. Specifically, we prove the asymptotic normality of the regression parameter and show that the convergence rate of the classification parameter can be arbitrarily close to $n^{-1}$ up to a logarithmic factor. We also conduct a simulation study to evaluate the finite sample performance of the MSPLE and apply the proposed method to real-world data.

The proposed method can be a useful tool for analyzing survival data when there exists a subgroup with a differentiated covariate effect.
Specifically, the indicator function combined with the estimated classification parameter provides clear criteria for judging whether certain individuals belong to the subgroup or not given their profiles.
In practice, however, it is suggested that data analysts test the existence of a subgroup in advance using the method such as \cite{wu2016subgroup} and \cite{kang2017subgroup}
because it is rare that the analysts are certain about the existence of a subgroup a priori.
Furthermore, the proposed method hinges on several important assumptions such as independent censoring and proportionality of the hazard.
The development of model diagnostics or more flexible models can be an interesting future research topic.

\clearpage
\appendix

\section{Proofs}

In what follows, we abbreviate $(T, Z, U, X)$ to $R$ for notational convenience. Define $P$ to be a distribution of $(\delta, T, W)$ and $P_n$ to be its empirical distribution. For a measurable function $f$, we write $P f := \int f dP$ and $P_n f := n^{-1} \sum_{i = 1}^n f(\delta_i, T_i, W_i)$. For each $\theta \in \Theta$, $n \in \mbb N$, and $t \in [0, \tau]$, we define the following measurable functions mapping $(\delta, T, W)$ into a real vector as
\begin{align*}
	&\Psi^{(1)}_n (\theta) := \delta (U' \gamma) (X / h_n) \mc K' ((V + X' \psi)/h_n), & &\Psi^{(2)}_n (\theta, t) := Y(t)e^{\eta_n (W, \theta)}, \\
	&\Psi^{(3)}_n (\theta, t) := Y(t)(U'\gamma)(X / h_n) \mc K' ((V + X' \psi)/h_n) e^{\eta_n (W, \theta)}, &
	&\Psi^{(4)} (\theta, t) := Y(t) e^{\eta (W, \theta)}.
\end{align*}
As we repeatedly use a combination of Assumptions \ref{a survival time and a censoring time}(b), \ref{the parameter space}(a) and (b), and \ref{covariates assumption}(a) to bound quantities below from zero and/or above from some finite positive constant, we combine them and refer to the combination as the boundedness assumption (BA) for brevity.

\subsection{Proofs of the Propositions}

\begin{proof}[Proof of Proposition \ref{identification}]
	First, we prove $(a)$ identification of $\theta_0$ and, subsequently,  show $(b)$ identification of $\lambda_0$.
	
	\underline{(a) Identification of $\theta_0$.}
	Observe that, for any $s \in (0, \tau)$,
	\begin{align}\label{identification theta 1}
		\mbb P(T \leq s, \delta = 1 | W) &= \mbb P(T^{\circ} \leq s, T^{\circ} \leq C | W) \notag \\
		&= \int_{t \leq s, t \leq c} F_{(T^{\circ}, C) | W} (dtdc) \notag \\
		&= \int_{0}^{s} \left( \int_{[t, \infty)} F_{C|W} (dc) \right) f_{T^{\circ} | W} (t | W) dt \notag\\
		&= \int_{0}^{s} \mbb P (C \geq t | W) \lambda_0 (t) e^{\eta (W, \theta_0)} e^{- e^{\eta (W, \theta_0)} \Lambda_0 (t)} dt,
	\end{align}
	where the third equality follows from Assumption \ref{a survival time and a censoring time}(a) and (d), and the last one follows from the definition of a hazard function.
	It follows from (\ref{identification theta 1}) and the fundamental theorem of calculus that, for all $t \in (0, \tau)$ except at most countable points,
	\begin{align}\label{identification theta 2}
		h(t) &=  \mbb P(C \geq t | W) \lambda_0 (t) e^{\eta (W, \theta_0)} e^{- e^{\eta (W, \theta_0)} \Lambda_0 (t)},
	\end{align}
	where $\Lambda (t)$ is defined in (\ref{Lambda}) and
	\[
		h(t) := \lim_{\ve \rightarrow 0} \frac{\mbb P(T \leq t + \ve, \delta = 1 | W) - \mbb P(T \leq t, \delta = 1 | W)}{\ve}.
	\]
	Meanwhile, for any $t \in (0, \tau)$,
	\begin{equation}\label{identification lambda 1}
		\mbb P(T \geq t | W) = \mbb P(T^{\circ} \geq t, C \geq t | W)
		= \mbb P(T^{\circ} \geq t | W)\mbb P(C \geq t | W)
		= e^{- e^{\eta(W, \theta_0)} \Lambda_0 (t)} \mbb P(C > t | W).
	\end{equation}
	By (\ref{identification theta 2}) and (\ref{identification lambda 1}), it holds that 
	\begin{align} 
		h(t) / \mbb P(T \geq t | W) = \lambda_0(t) e^{\eta (W, \theta_0)}
	\end{align}
	for all $t \in (0, \tau)$ except at most countable points, where the left-hand side is determined from the distribution of $(\delta, T, W)$.

	Now, assume that another parameter $(\lambda_1, \theta_1) \in \Pi \times \Theta$ generates the same distribution as the one under $(\theta_0, \lambda_0)$. Repeating the argument in the previous paragraph and by Assumption \ref{the parameter space}(b), $\lambda_0(t) e^{\eta (W, \theta_0)} = \lambda_1(t) e^{\eta (W, \theta_1)}$ for all $t \in (0, \tau) \ a.s.$ Pick $t' \in (0, \tau)$ such that $\lambda_0 (t')$ and $\lambda_1 (t')$ are both positive. Then, we have $\eta (W, \theta_0) - \eta (W, \theta_1) + c = 0 \ a.s.$ where $c := \log \lambda_0 (t') - \log \lambda_1 (t')$.
	We proceed to show $\mathbb E[\{\eta (W, \theta_0) - \eta (W, \theta_1) + c \}^{2}] > 0$ if $\theta_0 \neq \theta_1$, which is sufficient to show the identification of $\theta_0$. We follow the argument of the proof of Lemma 1 in \cite{seo2007smoothed}. Namely, we find a set $A$ with $\mbb P(A) > 0$ such that $\mathbb E[\{\eta (W, \theta_0) - \eta (W, \theta_1) + c \}^{2}, A] > 0$ where $\mathbb E[\widetilde X, A] := \int_{A} \tilde X(\omega) d \mbb P$ for a random variable $\widetilde X$. We split the proof into the following two cases.
	
	\underline{Case (i): $(\beta_0, \gamma_0) = (\beta_1, \gamma_1)$ and $\psi_0 \neq \psi_1$}. Define sets $A :=  \{ \omega \in \Omega: - X' \psi_0 \leq V < - X' \psi_1 \}$ and $B := \{ \omega \in \Omega: - X' \psi_1 \leq V < - X' \psi_0 \}$.
	It must be noted that $\mbb P (X'\psi_0 \neq X'\psi_1) > 0$ by Assumption \ref{covariates assumption}(c). It follows from Assumption \ref{covariates assumption}(b) that $\mbb P(A) > 0$ or $\mbb P(B) > 0$. Without loss of generality, suppose $\mbb P(A) > 0$. Then, observe that $\mathbb E[\{\eta (W, \theta_0) - \eta (W, \theta_1) + c \}^{2}, A \cup B] \geq \mathbb E[(U'\gamma_0 + c)^2, A] = \mathbb E[\mathbb E[(U'\gamma_0 + c)^2 | V, X], A] \geq \mathbb E[\gamma_0 Var(U | V, X) \gamma_0, A]$, which is positive by Assumptions \ref{the parameter space}(c) and \ref{covariates assumption}(d).
	
	\underline{Case (ii): $(\beta_0, \gamma_0) \neq (\beta_1, \gamma_1)$.} Define sets $A :=  \{ \omega \in \Omega: V + X'\psi_0 \geq 0, V + X'\psi_1 \geq 0 \}$ and $B := \{ \omega \in \Omega: V + X'\psi_0 < 0, V + X'\psi_1 < 0 \}.$
	Then, $\mbb P(A) > 0$ and $\mbb P(B) > 0$ by Assumption \ref{covariates assumption}(b).
	Similar to Case (i), we obtain $\mathbb E[\{\eta (W, \theta_0) - \eta (W, \theta_1) + c \}^{2}, A \cup B] = \mathbb E[\{Z'(\beta_0 - \beta_1) + U'(\gamma_0 - \gamma_1) + c \}^2, A] + \mathbb E[\{ Z'(\beta_0 - \beta_1) + c \}^2, B]$.
	It must be noted that when $\beta_0 = \beta_1$, $\mathbb E[\{Z'(\beta_0 - \beta_1) + U'(\gamma_0 - \gamma_1) + c \}^2, A] = \mathbb E[\{U'(\gamma_0 - \gamma_1) + c \}^2, A] \geq \mathbb E[(\gamma_0 - \gamma_1)' Var(U | V, X) (\gamma_0 - \gamma_1), A] > 0$ by Assumption \ref{covariates assumption}(d). 
	However, when $\beta_0 \neq \beta_1$, $\mathbb E[\{ Z'(\beta_0 - \beta_1) + c \}^2, B] \geq \mathbb E[(\beta_0 - \beta_1)'Var(Z | V, X)(\beta_0 - \beta_1), B] > 0$ again by Assumption \ref{covariates assumption}(d).
	Therefore, $\mathbb E[\{\eta (W, \theta_0) - \eta (W, \theta_1) + c \}^{2}, A \cup B] > 0$ holds.
	
	By cases (i) and (ii), the identification of $\theta_0$ follows.
	
	\underline{(b) Identification of $\lambda_0$.} 
	As observed in (a), $\lambda_0 (\cdot) e^{\eta (W, \theta_0)} = \lambda_1 (\cdot) e^{\eta (W, \theta_0)} \ a.s.,$ if $(\delta, T, W)$ has the same distribution under $(\theta_0, \lambda_0)$ and $(\theta_0, \lambda_1) \in \Theta \in \Pi$.
	Because $e^{\eta (W, \theta_0)} > 0$, $\lambda_0 (\cdot) = \lambda_1 (\cdot)$ follows.
	This completes the proof of Proposition \ref{identification}.
	\end{proof}

	\begin{proof}[Proof of Proposition \ref{consistency}]
		Our proof is based on the argument of \cite{van2002semiparametric}.
		Define $Q_n (\theta, \Lambda) := \prod_{i = 1}^n (e^{\eta_n (W_i, \theta)} \lambda_i)^{\delta_i} e^{- e^{\eta_n (W_i, \theta) \Lambda (T_i)}}$
		for $\theta \in \Theta$ and $\Lambda \in \Pi_n$, where $\Pi_n$ is a collection of all step functions that increase by non-negative $\lambda_i$ at $T_i$ for $\delta_i = 1$.
		Based on the arguments in section 5.3.2 and Example 9.13 of \cite{van2002semiparametric}, and strict increasingness of the logarithm, MSPLE $\widehat \theta_n$ has the following representation:
		$(\widehat \theta_n, \widehat \Lambda_{n, \widehat \theta_n}) := \argmax_{\theta \in \Theta, \Lambda \in \Pi_n} \log Q_n (\theta, \Lambda)$, where $\widehat \Lambda_{n, \widehat \theta_n}$ increases by $(n \widehat M_{n, \widehat \theta_n} (T_i))^{-1}$ at $T_i$ for $\delta_i = 1$ with $\widehat M_{n, \theta} (t)$ denoting $P_n \Psi^{(2)}_n (\theta, t)$.
		Further, define $M_{0, \theta} (t) := P \Psi^{(4)} (\theta, t)$, $\Lambda_{\theta} (t) := P \left( \delta 1 \{ T \leq t \} / M_{0, \theta} (T) \right)$, $\Lambda_{n, \theta} (t) := P_n \left( \delta 1 \{T \leq t \} / M_{0, \theta} (T) \right)$ and $\widehat \Lambda_{n, \theta} (t) := P_n \left( \delta 1 \{ T \leq t \} / \widehat M_{n, \theta} (T) \right)$.
		Letting  $\lambda_{\theta} (t) := \lambda_0 (t) M_{0, \theta_0} (t) / M_{0, \theta} (t)$, we can express $\Lambda_{\theta} (t)$ as follows:
		\begin{align}
			\Lambda_{\theta} (t) = P \int_{0}^{t} M_{0, \theta} (s)^{-1} d N(s)
			= P \int_{0}^{t} M_{0, \theta} (s)^{-1} Y(s) e^{\eta (W, \theta_0)} d \Lambda_0 (s) = \int_{0}^t \lambda_{\theta} (s) ds \notag,
		\end{align}	
		where $N(t) = 1 \{T \leq t, \delta = 1 \}$, and the second equality follows from Lemma \ref{martingale property}(b) and the third one from Fubini's Theorem.
		Letting $\widetilde \Lambda_n$ be a step function that increases by $(n M_{0, \theta_0} (T_i))^{-1}$ at $T_i$ for $\delta_i = 1$, it follows from $\log Q_n(\widehat \theta_n, \widehat \Lambda_{n, \widehat \theta_n}) \geq \log Q_n(\theta_0, \widetilde \Lambda_n)$ that 
		\begin{align} 
			P_n \delta (\eta_n (W, \widehat \theta_n) - \eta_n (W, \theta_0)) - P_n \left(e^{\eta_n (W, \widehat \theta_n)} \widehat \Lambda_{n, \widehat \theta_n} (T) - e^{\eta_n (W, \theta_0)} \widetilde \Lambda_n (T)\right) +
			 P_n \delta \log \frac{M_{0, \theta_0} (T)}{\widehat M_{n, \widehat \theta_n} (T)} \geq 0. \label{likelihood neq}
		\end{align}
		
		Suppose for a moment that the following $(a)$, $(b)$, $(c)$ and $(d)$ hold almost surely:
		\begin{align} 
			&(a) \quad \sup_{\theta \in \Theta} |P_n \delta \eta_n (W, \theta) - P \delta \eta(W, \theta)| \rightarrow 0,&
			&(b) \quad \sup_{\theta \in \Theta} |P_n e^{\eta_n (W, \theta)} \widehat \Lambda_{n, \theta} (T) - P e^{\eta (W, \theta)} \Lambda_{\theta} (T)| \rightarrow 0, \notag \\
			&(c) \quad | P_n e^{\eta_n (W, \theta_0)} \widetilde \Lambda_n (T) - P e^{\eta_0 (W, \theta_0)} \Lambda_0 (T) | \rightarrow 0, &
			&(d) \quad \sup_{\theta \in \Theta} \left| P_n \delta \log \frac{M_{0, \theta_0} (T)}{\widehat M_{n, \theta} (T)} - P \delta \log \frac{\lambda_{\theta} (T)}{\lambda_{0} (T)} \right| \rightarrow 0. \notag
		\end{align}
		Fix $\omega \in \{ \omega \in \Omega : (a), (b), (c) \ \text{and} \ (d) \ \text{hold, and} \ V(\omega) + X(\omega)'\psi \neq 0. \}$ and let $\theta_n := \widehat \theta_n (\omega)$.
		Take an arbitrary convergent subsequence $\theta_{n (k)}$ of $\theta_n$ and let $\theta_{\infty} \in \Theta$ denote its limit.
		Observe that $V (\omega) + X(\omega)'\psi \neq 0$ implies $\eta (W (\omega), \theta_{n(k)}) \rightarrow \eta (W (\omega), \theta_{\infty})$. % and $\mathbb P (V + X'\psi \neq 0) = 1$ by Assumption \ref{covariates assumption}(b). 
		Based on this observation, in conjunction with BA, (\ref{likelihood neq}) and the dominated convergence theorem, 
		\begin{align} 
			P\delta (\eta (W, \theta_\infty) - \eta (W, \theta_0)) - P \left( e^{\eta (W, \theta_{\infty})} \Lambda_0 (T) - e^{\eta (W, \theta_0)} \Lambda_{\theta_{\infty}} (T) \right) P \delta \log \frac{\lambda_{\theta_{\infty}} (T)}{\lambda_0 (T)} \geq 0. \label{likelihood diff limit}
		\end{align}
		$\lambda_{\theta_{\infty}} (\cdot)$ is left-continuous and continuous on $[0, \tau]$ except at most countable points. Thus, repeating the proof of Proposition \ref{identification}, $(\theta_0, \lambda_0)$ and $(\theta_{\infty}, \lambda_{\theta_{\infty}})$ generates distinct distributions of $(\delta, T, W)$ if $\theta_0 \neq \theta_{\infty}$. 
		Because 
		$f(\delta, t, w|\theta, \lambda) \propto \left( \lambda (t) e^{\eta (w, \theta)} \right)^{\delta} e^{- e^{\eta (w, \theta)} \Lambda (t)}$ is a likelihood function of $(\delta, T, W)$ as on page 901 of \cite{wei2018change}, the information inequality and (\ref{likelihood diff limit}) implies $\theta_{\infty} = \theta_0$. It follows from the compactness of $\Theta$ that $\theta_n \rightarrow \theta_0$ as $n \rightarrow \infty$. Therefore, $\widehat \theta_n$ converges to $\theta_0$ $a.s.$
		because $\mathbb P (V + X'\psi \neq 0) = 1$ by Assumption \ref{covariates assumption}(b).

		It remains to prove $a.s.$ convergence in $(a)$, $(b)$, $(c)$ and $(d)$.
		For $(a)$, the triangle inequality gives
		\begin{align} 
			\sup_{\theta \in \Theta} |P_n \delta \eta_n (W, \theta) - P \delta \eta(W, \theta)| \leq &\sup_{\theta \in \Theta} |P_n \delta \eta_n (W, \theta) - P_n \delta \eta (W, \theta)| + \sup_{\theta \in \Theta} |P_n \delta \eta (W, \theta) - P \delta \eta(W, \theta)|. \notag 
		\end{align}
		From BA, the first term on the right-hand side is bounded by $\sup_{\theta \in \Theta} P_n |\mathcal K ((V + X' \psi) / h_n ) - 1 \{ V + X'\psi \geq 0 \}|$, which converges to 0 $a.s.$ by Lemma 4 of \cite{horowitz1992smoothed}.
		The second term on the right-hand side converges to 0 $a.s.$ because $\{ \delta \eta (W, \theta) : \theta \in \Theta \}$ is Glivenko-Cantelli (e.g. see page 902 of \cite{wei2018change}).

		We move to $(b)$. From the triangle inequality, $\sup_{\theta \in \Theta} |P_n e^{\eta_n (W, \theta)} \widehat \Lambda_{n, \theta} (T) - P e^{\eta (W, \theta)} \Lambda_{\theta} (T)|$ is bounded by
		\begin{align}\label{four terms}
			&\sup_{\theta \in \Theta} |P_n e^{\eta_n (W, \theta)} \widehat \Lambda_{n, \theta} (T) - P_n e^{\eta (W, \theta)} \widehat \Lambda_{n, \theta} (T)| + \sup_{\theta \in \Theta} |P_n e^{\eta (W, \theta)} \widehat \Lambda_{n, \theta} (T) - P_n e^{\eta (W, \theta)} \Lambda_{n, \theta} (T)| \notag \\
			+ &\sup_{\theta \in \Theta} |P_n e^{\eta (W, \theta)} \Lambda_{n, \theta} (T) - P_n e^{\eta (W, \theta)} \Lambda_{\theta} (T)| + \sup_{\theta \in \Theta} |P_n e^{\eta (W, \theta)} \Lambda_{\theta} (T) - P e^{\eta (W, \theta)} \Lambda_{\theta} (T)|.
		\end{align}
		The first term in (\ref{four terms}) is bounded by $\mathcal C \sup_{\theta \in \Theta} P_n |\eta_n (W, \theta) - \eta (W, \theta)| \widehat M_{n, \theta} (\tau)^{-1}$ by BA and the fact that the exponential function is Lipschitz continuous on a bounded interval. 
		Note that $\widehat M_{n, \theta} (\tau)$ is bounded from below by some positive constant uniformly over $\Theta$ for sufficiently large $n$ $a.s.$ by BA. Hence, the first term in (\ref{four terms}) converges to $0$ $a.s.$ from Lemma 4 of \cite{horowitz1992smoothed}.
		By BA, the second term in (\ref{four terms}) is bounded by $\mathcal C \sup_{\theta \in \Theta, t \in [0, \tau]} |\widehat \Lambda_{n, \theta} (t) - \Lambda_{n, \theta} (t)| \leq \mathcal C \sup_{\theta \in \Theta} P_n |\widehat M_{n, \theta} (T)^{-1} - M_{0, \theta} (T)^{-1}| \leq \mathcal C \min \{ P_n Y(\tau), P Y(\tau) \}^{-2} \sup_{\theta \in \Theta} P_n | \widehat M_{n, \theta} (T) - M_{0, \theta} (T) |$,
		where the last inequality follows from the mean value theorem. 
		By the strong law of large numbers, $\min \{ P_n Y(\tau), P Y(\tau) \}^{-2} \rightarrow (P Y(\tau))^{-2} > 0 \ a.s.$ Subsequently, $\sup_{\theta \in \Theta} P_n | \widehat M_{n, \theta} (T) - M_{0, \theta} (T) | \leq \sup_{\theta \in \Theta, t \in [0, \tau]} |\widehat M_{n, \theta} (t) - M_{0, \theta} (t)| \rightarrow 0 \ a.s.,$ which follows from
		a similar argument to the first term and the fact that $\{e^{\eta (W, \theta)} Y(t) : \theta \in \Theta, t \in [0, \tau] \}$ is Glivenko-Cantelli (e.g. see page 902 of \cite{wei2018change}). As a result, the second term in (\ref{four terms}) converges to 0 $a.s.$
		Next, BA implies that the third term in (\ref{four terms}) is bounded by $\mathcal C \sup_{\theta \in \Theta, t \in [0, \tau]} |\Lambda_{n, \theta} (t) - \Lambda_{\theta} (t)|$.
		Because $M_{0, \theta} (\cdot)^{-1}$ is monotone increasing and bounded, an application of Lemma 9.11 of \cite{kosorok2008introduction} gives that $\{ M_{0, \theta} (T)^{-1} : \theta \in \Theta \}$ is Glivenko-Cantelli. Because $\{ 1 \{ T \leq t \}  : \in [0, \tau] \}$ is Glivenko-Cantelli,  $\{ \delta 1 \{ T \leq t \} / M_{0, \theta} (T) : \theta \in \Theta, t \in [0, \tau] \}$ is also Glivenko-Cantelli . Hence, the third term in (\ref{four terms}) converges to zero $a.s.$
		Lastly, $\Lambda_{\theta} (\cdot)$ is monotone increasing and bounded so that $\{ \Lambda_{\theta} (T) : \theta \in \Theta \}$ is Glivenko-Cantelli. Because $\{ \eta (W, \theta) : \theta \in \Theta \}$ is also Glivenko-Cantelli, their product $\{e^{\eta (W, \theta)} \Lambda_{\theta} (T) : \theta \in \Theta, t \in [0, \tau] \}$ is Glivenko-Cantelli by Theorem 3 of \cite{vaart2000preservation}. 
		This ensures the $a.s.$ convergence of the fourth term in (\ref{four terms}). 
		Therefore, we prove $(b)$.

		By a straightforward calculation, $\widetilde \Lambda_n (s) = P_n \delta 1 \{T \leq s \} / M_{0, \theta_0} (T)$. Thus, similarly to $(b)$, we obtain $a.s.$ convergence in $(c)$.

		Lastly, for $(d)$, recall that $\lambda_{\theta} (t) / \lambda_0 (t) = M_{0, \theta_0} (t) / M_{0, \theta} (t)$. It follows from the strong law of large numbers that $P_n \delta \log M_{0, \theta_0} (T) \rightarrow P \delta \log M_{0, \theta_0} (t) \ a.s.$
		Further, because $\log M_{0, \theta} (\cdot)$ is monotone decreasing and bounded, $\{ \log M_{0, \theta} (T) : \theta \in \Theta \}$ is Glivenko-Cantelli, and so is $\{ \delta \log M_{0, \theta} (T) : \theta \in \Theta  \}$. This shows the $a.s.$ convergence in $(d)$.
		This completes the proof.

		\end{proof}

		\begin{proof}[Proof Proposition \ref{roc psi}]
			By a straightforward calculation, $\nabla_{\psi} l^{*}_n (\theta)$ can be expressed as follows:
			\begin{align}\label{score_psi}
				&\frac{1}{n} \sum_{i = 1}^n \delta_i \left[ \frac{X_i}{h_n} (U_i' \gamma) \mc K' \left(\frac{V_i + X_i' \psi}{h_n} \right) - \frac{\frac{1}{n} \sum_{j = 1}^n Y_j (T_i) \frac{X_j}{h_n} (U_j' \gamma)  \mc K' \left(\frac{V_j + X_j' \psi}{h_n} \right) e^{\eta_n (W_j \theta)} }{\frac{1}{n} \sum_{j = 1}^n Y_j (T_i) e^{\eta_n (W_j \theta)}} \right] \notag \\
				= &P_n \Psi_n^{(1)} (\theta) - P_n \left( \delta \frac{\widetilde P_n \Psi_n^{(3)} (\theta, T)}{\widetilde P_n \Psi_n^{(2)} (\theta, T)} \right),
			\end{align}
			where $\widetilde P_n$ is an empirical measure with $T$ fixed. For example, $\widetilde P_n \Psi^{(4)} (\theta, T_i) = \frac{1}{n} \sum_{j = 1}^n Y_j (T_i) e^{\eta (W_j, \theta)}$.
			It follows from Lemma \ref{lem_roc} that
			\begin{align}\label{uni_convergence_score}
				\sup_{\theta \in \Theta_n} \left| \nabla_{\psi} l^{*}_n (\theta) - \tilde{S}_p^{\psi} (\psi) \right| = o_p (1),
			\end{align}
			where $\tilde{S}_p^{\psi} (\psi) := \int_{0}^{\tau} \mathbb E[Y(t) U' \gamma_0 (X/h_n) \mc K' ((V + X'\psi)/h_n)e^{Z'\beta_0}
			(e^{ U'\gamma_0 1\{V + X'\psi_0 \geq 0 \} } - e^{U'\gamma_0 \mc K((V + X' \psi)/h_n)}) ] d\Lambda_0 (t)$ and
			$\Theta_n := \{\theta \in \Theta : |\theta - \theta_0| \leq r_n \}$ are chosen such that $r_n \rightarrow 0$ and $\mbb P(|\hat{\theta}_n - 
			\theta_0| \geq r_n) \rightarrow 0$. Further, it must be noted that $\Theta_n \subset \Theta^{\circ}$ for sufficiently large $n$ by Assumption \ref{the parameter space}(a). Thus, for every $\ve > 0$, the first-order condition for MSPLE $\nabla_{\psi} l^{\ast}_n (\widehat \theta_n) = 0$ gives that
			\begin{align*}
				\mbb P\left(|\tilde S^{\psi}_p (\hat{\psi}_n)| \geq \ve \right) \leq &\mbb P\left(\hat{\theta}_n \notin \Theta_n\right) + \mbb P\left(|\nabla_{\psi} l^{*}_n (\hat \theta_n) - \tilde S^{\psi}_p (\hat{\psi}_n)| \geq \ve, \hat{\theta}_n \in \Theta_n \cap \Theta^{\circ}\right) \notag \\
				\leq &\mbb P\left(\hat{\theta}_n \notin \Theta_n\right) + \mbb P\left(\sup_{\theta \in \Theta_n} |\nabla_{\psi} l^{*}_n (\theta) - \tilde S^{\psi}_p (\psi)| \geq \ve \right),
			\end{align*}
			for sufficiently large $n$. Hence, $\tilde S^{\psi}_p (\hat{\psi}_n) = o_p (1)$ holds by (\ref{uni_convergence_score}).
			
			We will now derive the convergence rate of $\hat{\psi}_n$ from $\tilde S^{\psi}_p (\hat{\psi}_n) = o_p (1)$. Henceforth, let $(t', z, u, x, s)$ in integrals be a realization of $(T, Z, U, X, S)$. First, from the change of variable, $\rho = (s + x'(\psi - \psi_0)) / h_n$, and the mean value theorem along with Assumption \ref{kernel and smoothing 2}(b) and Proposition
			\ref{consistency}, for any $\psi$, the integrand in $\tilde{S}_p^{\psi} (\psi)$ is written as
			\begin{align}\label{step2_1}
				& \mathbb E[Y(t) U' \gamma_0 (X /h_n) \mc K' ((V + X' \psi)/h_n)e^{Z'\beta_0}
			(e^{ U'\gamma_0 1\{V + X'\psi_0 \geq 0 \} } - e^{U'\gamma_0 \mc K((V + X' \psi)/h_n)}) ] \notag \\
			= & \int dF_{R} \int 1\{t' \geq t \} u'\gamma_0 x \mc K'(\rho) e^{z' \beta_0} \left(e^{u'\gamma_0 1\{\rho h_n - x' (\psi - \psi_0) \geq 0 \}} 
			- e^{u' \gamma_0 \mc K(\rho)} \right) f_{S|R} (\rho h_n - x' (\psi - \psi_0) | r) d\rho \notag \\
			= & \int dF_{R} \int 1\{t' \geq t \} u'\gamma_0 x \mc K'(\rho) e^{z' \beta_0} \left(e^{u'\gamma_0 1\{\rho h_n - x' (\psi  - \psi_0) \geq 0 \}} 
			- e^{u' \gamma_0 \mc K(\rho)} \right) f_{S|R}(0 | r) d\rho + O(h_n + |\psi - \psi_0|),
			\end{align}
			uniformly in $t \in [0, \tau]$ as $n \rightarrow \infty$ and $|\psi - \psi_0| \rightarrow 0$. 
			By a straightforward integral calculation, the leading term on the right-hand side of (\ref{step2_1}) equals
			\begin{align}\label{step2_2}
			&\int 1\{t' \geq t \} x e^{z' \beta_0} f_{S|R}(0 | r) u' \gamma_0 \left[\int_{- \infty}^{\frac{x' (\psi  - \psi_0)}{h_n}} \mc K'(\rho) d\rho + \int_{\frac{ x' (\psi - \psi_0)}{h_n}}^{\infty} \mc K'(\rho) e^{u'\gamma_0} d\rho
			- \int_{- \infty}^{\infty} \mc K'(\rho) e^{u'\gamma_0 \mc K(\rho)} d\rho \right] dF_{R}  \notag \\
			=&  \int 1 \{t' \geq t \} x e^{z'\beta_0} f_{S|R}(0 | r) \left[\mc K(x' (\psi - \psi_0)/h_n) u' \gamma_0 (1 - e^{u'\gamma_0}) +
			e^{u'\gamma_0} (u'\gamma_0 - 1) + 1  \right]  dF_{R}.
			\end{align}
			Let $\psi_n := (\psi - \psi_0) / h_n$, $H(x'\psi_n, u, z, t, t') := 1 \{t' \geq t \} e^{z'\beta_0} \left[\mc K(x' \psi_n) u' \gamma_0 (1 - e^{u'\gamma_0}) +
			e^{u'\gamma_0} (u'\gamma_0 - 1) + 1  \right]$ and $dG(r, t) = f_{S|R}(0 | r) d F_{R}d\Lambda_0 (t)$.
			Then, by (\ref{step2_1}) and (\ref{step2_2}),
			\begin{align*}
				\tilde{S}_p^{\psi} (\psi) = \int_0^{\tau} \int x H(x'\psi_n, u, z, t, t') d G(r, t) + O(h_n + |\psi - \psi_0|).
			\end{align*}
			Since $\widehat \psi_n \rightarrow_p \psi_0$, evaluating the above inequality at $\psi = \widehat \psi_n$ gives
			\begin{align*} 
				\tilde{S}_p^{\psi} (\widehat \psi_n) = \int_0^{\tau} \int x H(x'\ddot{\psi}_n, u, z, t, t') d G(r, t) + o_p (1),
			\end{align*}
			where $\ddot{\psi}_n := (\widehat \psi_n - \psi_0)/h_n$. Since $\tilde{S}_p^{\psi} (\widehat \psi_n) = o_p (1)$, multiplying both sides of the above equation by $\ddot \psi_n / \| \ddot \psi_n \|$
			gives
			\begin{align}\label{score equation}
				\int_0^{\tau} \int x' \ddot \psi_n / \| \ddot \psi_n \| H(x'\ddot{\psi}_n, u, z, t, t') d G(r, t) = o_p (1).
			\end{align}
			
			Based on (\ref{score equation}), we prove $\ddot{\psi}_n  = O_p (1)$. Specifically, we show that
			\begin{align}\label{uniform bound}
				\sup_{\ddot \psi: \| \ddot \psi \| > M} \int_0^{\tau} \int x' \ddot \psi / \| \ddot \psi \| H(x'\ddot{\psi}, u, z, t, t') d G(r, t) < - \delta,
			\end{align}
			for some $M > 0$ and $\delta > 0$. Then, $\ddot \psi_n = O_p (1)$ follows from (\ref{score equation}) and (\ref{uniform bound}).

			Split the integral in (\ref{uniform bound}) into three parts as $\mathcal B_{1} + \mathcal B_{2} + \mathcal B_{3}$, where
			\begin{align*}
				\mathcal B_{1} &= \int_0^{\tau} \int_{x'\ddot{\psi} \geq \sqrt{M}} x' \ddot \psi / \| \ddot \psi \| H(x'\ddot{\psi}, u, z, t, t') d G(r, t), \\
				\mathcal B_{2} &= \int_0^{\tau} \int_{x'\ddot{\psi} \leq - \sqrt{M}} x' \ddot \psi / \| \ddot \psi \| H(x'\ddot{\psi}, u, z, t, t') d G(r, t), \\
				\mathcal B_{3} &= \int_0^{\tau} \int_{|x'\ddot{\psi} | < \sqrt{M}} x' \ddot \psi / \| \ddot \psi \| H(x'\ddot{\psi}, u, z, t, t') d G(r, t).
			\end{align*}
			First, we have $\sup_{\ddot \psi: \| \ddot \psi \| > M} |\mathcal B_3| \rightarrow 0$ as $M \rightarrow \infty$ because $1\{ |x'\ddot{\psi} | < \sqrt{M} \} x' \ddot \psi / \| \ddot \psi \| \rightarrow 0$ uniformly over $\{ \ddot \psi : \| \ddot \psi \| > M \}$, and the support of $X$ and 
			$H(x'\ddot{\psi}_n, u, z, t, t')$ is uniformly bounded over its support.
			We proceed to consider $\mathcal B_{1} + \mathcal B_{2}$. Observe that we have 
			$\mathcal K (x' \ddot \psi) \rightarrow 1$ as $x'\ddot \psi \rightarrow \infty$, and hence $H(x'\ddot{\psi}, u, z, t, t') \rightarrow 1 \{ t' \geq t \} e^{z'\beta_0} (u'\gamma_0 - e^{u'\gamma_0} + 1)$ uniformly over $\{ x, \ddot \psi : \| \ddot \psi \| > M \ \text{and} \ x'\ddot \psi > \sqrt{M}  \}$ and the support of $(x, u, z, t, t')$, as $M \rightarrow \infty$. Consequently,
			\begin{align}\label{B_1}
				\sup_{\ddot \psi : \| \ddot \psi \| > M} \left| \mathcal B_{1} - \int_{0}^{\tau} \int_{x'\ddot \psi \geq \sqrt{M}} x' \ddot \psi / \| \ddot \psi \| 1 \{ t' \geq t \} e^{z'\beta_0} (u'\gamma_0 - e^{u'\gamma_0} + 1) dG(r, t) \right| = o(1),
			\end{align}
			as $M \rightarrow \infty$ because $x' \ddot \psi / \| \ddot \psi \|$ is bounded.
			Let $c_1, c_2 > 0$ be such that $\mathbb P (|X'e| > c_1) > c_2$, for any $\| e \| = 1$, as in Lemma \ref{lem: matrix}. Observe that $s - e^s + 1 \leq 0$ for any $s \in \mathbb R$ and $\{x: x'\ddot \psi \geq \sqrt{M} \} \supset \{x: x' \ddot \psi \geq c_1 \|\ddot \psi \| \}$, when $M$ is sufficiently large and $\|\ddot \psi \| > M$.
			Then, for sufficiently large $M$, the integral in (\ref{B_1}) is no larger than
			\begin{align}\label{B_1 leading}
				&c_1 \int_{0}^{\tau} \int_{x'\ddot \psi / \| \ddot \psi  \|  \geq c_1 } 1 \{ t' \geq t \} e^{z'\beta_0} (u'\gamma_0 - e^{u'\gamma_0} + 1) dG(r, t) \notag \\
			= &c_1 \int_{0}^{\tau} \int 1 \{ X'\ddot \psi / \| \ddot \psi  \|  \geq c_1 \} \mathbb E [Y(\tau) e^{z'\beta_0} (u'\gamma_0 - e^{u'\gamma_0} + 1) f_{S|R} (0 | R)|X] d\mathbb P d\Lambda_0 (t),
			\end{align}
			for any $\ddot \psi \in \{ \ddot \psi : \| \ddot \psi \| > M \}$.
			By a similar argument to $\mathcal B_{1}$, we have
			\begin{align}\label{B_2}
				\sup_{\ddot \psi : \| \ddot \psi \| > M } \left| \mathcal B_{2} - \int_{0}^{\tau} \int_{x'\ddot \psi \leq - \sqrt{M}} x' \ddot \psi / \| \ddot \psi \| 1 \{ t' \geq t \} e^{z'\beta_0} (e^{u'\gamma_0} (u'\gamma_0 - 1) + 1) dG(r, t) \right| = o(1),
			\end{align}
			as $M \rightarrow \infty$. Since $e^s(s - 1) + 1 \geq 0$ for any $s \in \mathbb R$, for sufficiently large $M$, the integral in (\ref{B_2}) is no larger than
			\begin{align}\label{B_2 leading}
				- c_1 \int_{0}^{\tau} \int 1 \{ X'\ddot \psi / \| \ddot \psi  \|  \leq - c_1 \} \mathbb E [Y(\tau) e^{z'\beta_0} (e^{u'\gamma_0} (u'\gamma_0 - 1) + 1) f_{S|R} (0 | R)|X] d\mathbb P d\Lambda_0 (t),
			\end{align}
			for any $\ddot \psi \in \{ \ddot \psi : \| \ddot \psi \| > M \}$. Let $D(X)$ denote 
			\begin{align*} 
				\max \left\{ \mathbb E [Y(\tau) e^{z'\beta_0} (u'\gamma_0 - e^{u'\gamma_0} + 1) f_{S|R} (0 | R) | X],
					- \mathbb E [ Y(\tau) e^{z'\beta_0} (e^{u'\gamma_0} (u'\gamma_0 - 1) + 1) f_{S|R} (0 | R) | X] \right\}.
			\end{align*}
			Then, by (\ref{B_1}), (\ref{B_1 leading}), (\ref{B_2}), and (\ref{B_2 leading}), we have
			\begin{align}\label{bound B_1 + B_2}
				\sup_{\ddot \psi : \| \ddot \psi \| > M} (\mathcal B_1 + \mathcal B_2) \leq c_1 \left( \sup_{\ddot \psi : \| \ddot \psi \| > M} \int_{0}^{\tau} \int 1 \{ | X'\ddot \psi / \| \ddot \psi  \| | \geq  c_1 \} D(X) d\mathbb P d\Lambda_0 (t) \right) + o(1),
			\end{align}
			as $M \rightarrow \infty$.
			Let $c_3 \in (0, 1)$ be sufficiently large such that $c_2 + c_3 > 1$. Then, since $D(X) < 0$ $\mathbb P$-$a.s.$ from Assumption \ref{moment condition 2}(b), 
			$\mathbb P ( D(X) \leq c_4 ) \geq c_3$ for some $c_4 < 0$. Therefore, for any $\ddot \psi \in \{\ddot \psi : \| \ddot \psi \| > M \}$,
			\begin{align*}
				\int  1 \{ | X'\ddot \psi / \| \ddot \psi  \| | \geq  c_1 \} D(X) d\mathbb P 
				\leq &c_4 \mathbb P \left( \left\{ | X'\ddot \psi / \| \ddot \psi  \| | \geq  c_1  \right\} \cap \left\{D(X) \leq c_4 \right\} \right) \notag \\
				\leq &c_4  \left[ \mathbb P \left( | X'\ddot \psi / \| \ddot \psi  \| | \geq  c_1 \right) + \mathbb P \left( D(X) \leq c_4 \right) - 1  \right] \notag \\ 
				\leq &c_4 (c_1 + c_2 - 1),
			\end{align*}
			because $\left\| \ddot \psi / \| \ddot \psi\| \right\| = 1$. Since the right-hand side does not depend on the value of $\ddot \psi$, we establish (\ref{uniform bound}) from
			(\ref{bound B_1 + B_2}) and $\lim_{M \rightarrow \infty} \sup_{\ddot \psi: \| \ddot \psi \| > M} |\mathcal B_3| = 0$. This completes the proof.
			
			\end{proof}

			\begin{proof}[Proof of Proposition \ref{a score}] First, we introduce the following notation:
				\begin{align*}
					\begin{cases}
						S^{(0)}_n (t, \theta) := n^{-1} \sum_{j = 1}^n Y_j (t) e^{\eta (W_j, \theta)} \\
						S^{(1)}_n (t, \theta) := n^{-1} \sum_{j = 1}^n Y_j (t) \phi(\theta)_j e^{\eta (W_j, \theta)}
					\end{cases}
					\begin{cases}
						\widetilde S^{(0)}_n (t, \theta) := n^{-1} \sum_{j = 1}^n Y_j (t) e^{\eta_n (W_j, \theta)} \\
						\widetilde S^{(1)}_n (t, \theta) := n^{-1} \sum_{j = 1}^n Y_j (t) \phi^{(n)}(\theta)_j e^{\eta_n (W_j, \theta)}
					\end{cases}
				\end{align*}
				and henceforth, for notational convenience, we write $\phi_i$ and $\phi^{(n)}_i$ for $\phi (\theta_0)_i$ and $\phi^{(n)} (\theta_0)_i$, respectively.
				First, we observe that 
				\begin{align}\label{score expression}
					n^{1/2} \nabla_{\xi}  l^*_{n} (\theta_0) = n^{- 1/2} \sum_{i = 1}^n \delta_i \left[ \phi^{(n)}_i - \frac{\widetilde S_n^{(1)} (T_i, \theta_0)}{\widetilde S_n^{(0)} (T_i, \theta_0)} \right].
				\end{align}
				It must be noted that, by BA,
				\begin{align*}
					\left\|n^{- 1/2} \sum_{i = 1}^n \delta_i \phi^{(n)}_i - n^{- 1/2} \sum_{i = 1}^n \delta_i \phi_i \right\| \leq \mathcal C n^{- 1/2} \sum_{i = 1}^n \left| \mathcal K \left( \frac{V_i + X_i' \psi_0}{h_n} \right) - 1 \{ V_i + X_i'\psi_0 \geq 0 \} \right|,
				\end{align*}
				and the right-hand side is $o_p (1)$ by Lemma \ref{error indicator kernel}. Similarly, a straight-forward calculation yields 
				\begin{align}\label{kernel ind bound}
					\left\| n^{- 1/2} \sum_{i = 1}^n \delta_i \frac{\widetilde S^{(1)}_n (T_i, \theta_0)}{\widetilde S^{(0)}_n (T_i, \theta_0)} -  n^{- 1/2} \sum_{i = 1}^n \delta_i \frac{ S^{(1)}_n (T_i, \theta_0)}{ S^{(0)}_n (T_i, \theta_0)}\right\| \leq \frac{\mathcal C n^{- 1/2} \sum_{i = 1}^n \left| \mathcal K \left( \frac{V_i + X_i' \psi_0}{h_n} \right) - 1 \{ V_i + X_i'\psi_0 \geq 0 \} \right|}{\left( n^{-1} \sum_{j = 1}^n Y_j (\tau)\right)^2}.
				\end{align}
				Since $n^{-1} \sum_{j = 1}^n Y_j (\tau) \rightarrow_p \mathbb P(T \geq \tau) > 0$ by the law of large numbers, Lemma \ref{error indicator kernel} implies that the right-hand side of (\ref{kernel ind bound}) is $o_p (1)$. Therefore, from (\ref{score expression}), we obtain 
				\begin{align}\label{score approximation}
					n^{1/2} \nabla_{\xi}  l^*_{n} (\theta_0) = n^{- 1/2} \sum_{i = 1}^n \delta_i \left[ \phi_i - \frac{S_n^{(1)} (T_i, \theta_0)}{ S_n^{(0)} (T_i, \theta_0)} \right] + o_p (1).
				\end{align}
				Now, the problem has boiled down to an asymptotic theory for a standard Cox model.
				Namely, it follows from Theorem 8.2.1 and 8.4.1 of \cite{fleming2011counting} that 
				\begin{align*} 
					n^{- 1/2} \sum_{i = 1}^n \delta_i \left[ \phi_i - \frac{S_n^{(1)} (T_i, \theta_0)}{ S_n^{(0)} (T_i, \theta_0)} \right] \rightarrow_d N(0, \mathcal I),
				\end{align*}
				which, when combined with (\ref{score approximation}), completes the proof.
				\end{proof}
				
				\begin{proof}[Proof of Proposition \ref{convergence matrices}] 
					In the proof, $\theta$ in a matrix-valued function may take different value row by row.
					Observe that
				\[
				\nabla_{\xi \xi'} l_n^{*} (\theta) = \frac{1}{n} \sum_{i = 1}^n \delta_i \left( \frac{\widetilde S_n^{(1)} (T_i, \theta)^{\otimes 2} - \widetilde S_n^{(0)} (T_i, \theta) \widetilde S_n^{(2)} (T_i, \theta)}{\widetilde S_n^{(0)} (T_i, \theta)^2} \right),
				\]
				where $\widetilde S_n^{(2)} (t, \theta) := n^{-1} \sum_{j = 1}^n Y_j (t) \phi^{(n)} (\theta)_j^{\otimes 2} e^{\eta_n (W_j, \theta)}$. 
				By a straightforward calculation, we have 
				\begin{align}\label{approximation ind kernel 1}
					\sup_{\theta \in \Theta} \left|\frac{1}{n} \sum_{i = 1}^n \delta_i \frac{\widetilde S_n^{(1)} (T_i, \theta)^{\otimes 2}}{\widetilde S_n^{(0)} (T_i, \theta)^2} - \frac{1}{n} \sum_{i = 1}^n \delta_i \frac{S_n^{(1)} (T_i, \theta)^{\otimes 2}}{S_n^{(0)} (T_i, \theta)^2} \right| \leq \frac{\mathcal C \sup_{\theta \in \Theta} \frac{1}{n} \sum_{i = 1}^n \left| \mathcal K (\frac{V_i + X_i'\psi}{h_n}) - 1\{V_i + X_i'\psi \geq 0 \} \right|}{\left(\frac{1}{n} \sum_{i = 1}^n Y_i (\tau) \right)^4},
				\end{align} 
				where we use BA and the fact that for any real number $a, b$ and any real-valued matrix $A, B$, $|a A - b B| \leq |a| |A - B| + |a - b| |B|$. Since $\frac{1}{n} \sum_{i = 1}^n Y_i (\tau) \rightarrow_p \mathbb P (T_i \geq \tau) > 0$, by the law of large numbers,
				 it follows from Lemma 4 of \cite{horowitz1992smoothed} that the right-hand side of (\ref{approximation ind kernel 1}) is $o_p(1)$. A similar argument gives
				\begin{align*} 
					\sup_{\theta \in \Theta} \left|\frac{1}{n} \sum_{i = 1}^n \delta_i \frac{\widetilde S_n^{(2)} (T_i, \theta)}{\widetilde S_n^{(0)} (T_i, \theta)} - \frac{1}{n} \sum_{i = 1}^n \delta_i \frac{S_n^{(2)} (T_i, \theta)}{S_n^{(0)} (T_i, \theta)} \right| \leq \frac{\mathcal C \sup_{\theta \in \Theta} \frac{1}{n} \sum_{i = 1}^n \left| \mathcal K (\frac{V_i + X_i'\psi}{h_n}) - 1\{V_i + X_i'\psi \geq 0 \} \right|}{\left(\frac{1}{n} \sum_{i = 1}^n Y_i (\tau) \right)^2},
				\end{align*}
				where $S_n^{(2)} (t, \theta) := n^{-1} \sum_{j = 1}^n Y_j (t) \phi (\theta)_j^{\otimes 2} e^{\eta (W_j, \theta)}$, and the right-hand side is $o_p (1)$. Therefore,
				\begin{align}\label{hessian approximation}
					\nabla_{\xi \xi'} l_n^{*} (\Bar \theta_n) = \frac{1}{n} \sum_{i = 1}^n \delta_i \left( \frac{S_n^{(1)} (T_i, \Bar \theta_n)^{\otimes 2} - S_n^{(0)} (T_i, \Bar \theta_n) S_n^{(2)} (T_i, \Bar \theta_n)}{S_n^{(0)} (T_i, \Bar \theta_n)^2} \right) + o_p (1).
				\end{align}
				Split $\theta$ into two parts as $\theta = (\xi, \psi)$. Then, similar to (\ref{approximation ind kernel 1}), we have
				\begin{align}\label{bound error ind}
					\left| \frac{1}{n} \sum_{i = 1}^n \delta_i \frac{S_n^{(1)} (T_i, \Bar \theta_n)^{\otimes 2}}{S_n^{(0)} (T_i, \Bar \theta_n)^2} - \frac{1}{n} \sum_{i = 1}^n \delta_i \frac{S_n^{(1)} (T_i, (\Bar \xi_n, \psi_0))^{\otimes 2}}{S_n^{(0)} (T_i, (\Bar \xi_n, \psi_0))^2}\right|
					\leq \frac{\mathcal C \frac{1}{n} \sum_{i = 1}^n |1 \{ V_i + X_i'\Bar \psi_n \geq 0 \} - 1 \{V_i + X_i'\psi_0 \geq 0 \}|}{\left(\frac{1}{n} \sum_{i = 1}^n Y_i (\tau) \right)^4}.
				\end{align}
				Since $\frac{1}{n} \sum_{i = 1}^n Y_i (\tau) \rightarrow_p \mathbb P (T_i \geq \tau) > 0$ by the law of large numbers and $\Bar \psi_n \rightarrow_p \psi_0$, Lemma \ref{error two indicator}
				implies that the right-hand side of (\ref{bound error ind}) is $o_p (1)$.
				By a similar argument, we obtain
				\begin{align*}
					\left| \frac{1}{n} \sum_{i = 1}^n \delta_i \frac{S^{(2)}_n (T_i, \Bar \theta_n)}{S^{(0)}_n (T_i, \Bar \theta_n)} - \frac{1}{n} \sum_{i = 1}^n \delta_i \frac{S^{(2)}_n (T_i, (\Bar \xi_n, \psi_0))}{S^{(0)}_n (T_i, (\Bar \xi_n, \psi_0))} \right| = o_p (1).
				\end{align*}
				Therefore, (\ref{hessian approximation}) and the subsequent argument implies that
				\begin{align}\label{hessian approximation 2}
					\nabla_{\xi \xi'} l_n^{*} (\Bar \theta_n) = \frac{1}{n} \sum_{i = 1}^n \delta_i \left( \frac{S_n^{(1)} (T_i, (\Bar \xi_n, \psi_0))^{\otimes 2} - S_n^{(0)} (T_i, (\Bar \xi_n, \psi_0)) S_n^{(2)} (T_i, (\Bar \xi_n, \psi_0))}{S_n^{(0)} (T_i, (\Bar \xi_n, \psi_0))^2} \right) + o_p (1).
				\end{align}
				We now apply the result for a basic Cox model. Specifically, by Theorem 8.2.2 and 8.4.1 of \cite{fleming2011counting}, we establish
				\begin{align}\label{hessian approximation 3}
					\frac{1}{n} \sum_{i = 1}^n \delta_i \left( \frac{S_n^{(1)} (T_i, (\Bar \xi_n, \psi_0))^{\otimes 2} - S_n^{(0)} (T_i, (\Bar \xi_n, \psi_0)) S_n^{(2)} (T_i, (\Bar \xi_n, \psi_0))}{S_n^{(0)} (T_i, (\Bar \xi_n, \psi_0))^2} \right) \rightarrow_p - \mathcal I.
				\end{align}
				The stated result follows from (\ref{hessian approximation 2}) and (\ref{hessian approximation 3}).
				\end{proof}
				
				\begin{proof}[Proof of Proposition \ref{asy msple xi}]
				By Assumption \ref{the parameter space}(a), Proposition \ref{consistency}, Assumption \ref{pd variance matrix}, Proposition \ref{convergence matrices}, and an equation (\ref{expansion xi 0}), we have, with probability approaching one,
				\begin{align}\label{expansion xi}
					n^{1/2} (\widehat \xi_n - \xi_0) = - \nabla_{\xi \xi'} l^{\ast}_n (\Bar \theta_n)^{-1} n^{1/2} \nabla_{\xi} l^{\ast}_n (\theta_0) - \nabla_{\xi \xi'} l^{\ast}_n (\Bar \theta_n)^{-1} \nabla_{\xi \psi'} l^{\ast}_n (\Bar \theta_n) n^{1/2} (\widehat \psi_n - \psi_0).
				\end{align}
				By Propositions \ref{a score} and \ref{convergence matrices}, the first term on the right-hand side of (\ref{expansion xi}) converges to $N(0, \mathcal I^{-1})$ in distribution. 
				Hence, the stated result holds if the second term on the right-hand side of (\ref{expansion xi}) is $o_p(1)$. 
				Since $n^{1/2} (\widehat \psi_n - \psi_0) = o_p (1)$ from Assumption \ref{kernel and smoothing 2}(c) and Proposition \ref{roc psi}, and $\nabla_{\xi \xi'} l^{\ast}_n (\Bar \theta_n)^{-1} = O_p (1)$ from Proposition \ref{convergence matrices},
				the desired result follows if we show $\nabla_{\xi \psi'} l^{\ast}_n (\Bar \theta_n) = O_p (1)$.
				Define $\varphi^{(n)} (\theta)_i := (U_i ' \gamma) X_i / h_n \mc K' ((V_i + X_i' \psi)/ h_n)$ and
				\[
					\Upsilon^{(n)} (\theta)_i :=
					\begin{pmatrix}
						O_{p_1 \times q} \\
						U_i X_i' /h_n \mc K' ((V_i + X_i' \psi)/h_n)
					\end{pmatrix}.
				\]
				Then, observe that
				\begin{align*}
					&\nabla_{\xi \psi'} l^{\ast}_n (\theta) = \frac{1}{n} \sum_{i = 1}^n \delta_i \left[ \Upsilon^{(n)} (\theta)_i - \frac{n^{-1} \sum_{j = 1}^n Y_j (T_i) \left( \Upsilon^{(n)} (\theta)_j +  \phi^{(n)} (\theta)_j \varphi^{(n)} (\theta)_j' \right) e^{\eta_n (W_j, \theta)}}{\tilde S_n^{(0)} (T_i, \theta)} \right. \notag \\
					& \left. + \frac{\left(n^{-1} \sum_{j = 1}^n Y_j (T_i) \phi^{(n)} (\theta)_j e^{\eta_n (W_j, \theta)} \right) \left( n^{-1} \sum_{j = 1}^n Y_j (T_i) \varphi^{(n)} (\theta)_j e^{\eta_n (W_j, \theta)} \right)' }{\tilde S_n^{(0)} (T_i, \theta)^2}\right].
				\end{align*}
				Hence, by BA, $\sup_{\theta \in \Theta} |\nabla_{\xi \psi'} l^{\ast}_n (\theta)|$ is bounded by
				\[
					\mc C \left( \sup_{\theta \in \Theta} \frac{1}{n} \sum_{i = 1}^n \frac{\mc K' ((V_i + X_i' \psi)/h_n)}{h_n} + \frac{\sup_{\theta \in \Theta} \frac{1}{n} \sum_{i = 1}^n \frac{\mc K' ((V_i + X_i' \psi)/h_n)}{h_n}}{n^{-1} \sum_{i = 1} Y_i (\tau)}
					+ \frac{\sup_{\theta \in \Theta} \frac{1}{n} \sum_{i = 1}^n \frac{\mc K' ((V_i + X_i' \psi)/h_n)}{h_n}}{\left(n^{-1} \sum_{i = 1} Y_i (\tau) \right)^2} \right).
				\]
				By the law of large numbers, $n^{-1} \sum_{i = 1} Y_i (\tau) \rightarrow_p \mbb P (T \geq \tau) > 0$. Furthermore,
				\begin{align*}
					&\sup_{\theta \in \Theta} \frac{1}{n} \sum_{i = 1}^n \frac{\mc K' ((V_i + X_i' \psi)/h_n)}{h_n} \notag \\ \leq &\sup_{\theta \in \Theta} \left|  \frac{1}{n} \sum_{i = 1}^n \frac{\mc K' ((V_i + X_i' \psi)/h_n)}{h_n}
					- E\left[ \frac{\mc K' ((V + X' \psi)/h_n)}{h_n} \right] \right| + \sup_{\theta \in \Theta} E\left[  \frac{\mc K' ((V + X' \psi)/h_n)}{h_n}   \right].
				\end{align*}
				By the proof of Lemma \ref{lem_roc}(a), the first term on the right-hand side is $o_p(1)$. Additionally, the second term is bounded uniformly in $n$ by (\ref{kernel inequality}). 
				Therefore, $\nabla_{\xi \psi'} l^{\ast}_n (\Bar \theta_n) = O_p (1)$. This completes the proof.
				\end{proof}

\subsection{Auxiliary Results}

In this appendix, we will state several auxiliary results used in Appendix A.1.

\begin{lemma}\label{martingale property}
	Suppose the assumptions of Proposition \ref{consistency} hold. Define $N_i (t) := 1\{T_i \leq t, \delta_i = 1 \}$ and $N^{U}_i (t) := 1\{ T_i \leq t, \delta_i = 0 \}$ for $t \in [0, \tau]$ and $i = 1, \dots, n$, with $n \in \mbb N$. For fixed $n \in \mbb N$, define a filtration $\{ \mc F_{n, t} : t \in [0, \tau] \}$ as $\mc F_{n, t} := \sigma(N_i (u), N_i^{U} (u), W_i : u \in [0, t], i = 1, \dots, n)$. Then, \\ (a) $\mc F_{n, t}$ is right-continuous. \\ (b) $M_i (t) := N_i (t) - \int_{0}^{t} Y_i (u) e^{\eta_0 (W_i, \theta_0)} d \Lambda_0 (u)$ is an $\mc F_{n, t}$-martingale.
\end{lemma}
	
\begin{proof}[Proof]
	This lemma is the basic result of the Cox model. Part (a) follows from T26 on page 304 of \cite{bremaud}. For part (b), \cite{fleming2011counting} show the stated result when the model has no covariate $W$. The proof with covariate $W$ follows the argument of the proof of Theorem 1.3.1 of \cite{fleming2011counting}.
\end{proof}

\begin{lemma}\label{empirical process}
	Suppose that the assumptions of Proposition \ref{roc psi} hold. For a class of real-valued functions defined on $(S, \mathcal S)$, $\mathcal F$, a probability measure $Q$, and real numbers $\ve, r > 0$, let $N(\ve, \mathcal F, L_r(Q))$ be a covering number (see page 18 of \cite{kosorok2008introduction} for the definition of a covering number).
	Then the following statements hold, where the supremum $\sup_Q$ is taken over all the finitely supported probability measure $Q$ on $(S, \mathcal S):$ \\
	(a) Let $\mathcal F$ be a VC class of functions with a constant envelope $F$ (refer to pages 156 and 157, and 18 of \cite{kosorok2008introduction} for the definitions of VC class and its related concept, and an envelope, respectively).
	Then there exist $0 < U, A, \nu < \infty$ such that
	$\sup_Q N (\ve U, \mathcal F, L_2 (Q)) \leq (A / \ve)^{\nu}$, for any $0 < \ve < 1$ where $U, A, \nu$ depend only on a VC dimension of $\mathcal F$. \\
	(b) Let $\mathcal F$ be a finite-dimensional vector space of measurable functions. Then $\mathcal F$ is a VC class with its VC dimension being equal to or less than $dim(\mathcal F) + 2$ where $dim(\mathcal F)$
	is a dimension of $\mathcal F$. \\
	(c) Let $\mathcal F$ be a VC class of functions and $\phi$ be a nondecreasing function. Then $\phi(\mathcal F)$ is a VC class with its VC dimension being equal to or less than that of $\mathcal F$. \\
	(d) Let $\mathcal F_{i}$ be a class of functions with a constant envelope $F_i$ such that there exist $0 < U_i, A_i, \nu_i < \infty$  satisfying 
	$\sup_Q N(\ve U_i, \mathcal F_{i}, L_2 (Q)) \leq (A_i / \ve)^{\nu_i}$, for any $0 < \ve < 1 \ (i = 1, 2)$. Then there exist $0 < U^*_1, U^*_2, A^*_1, A^*_2, \nu^*_1, \nu^*_2 < \infty$ such that
	$\sup_Q N(\ve U^*_1, \mathcal F_1 \mathcal F_2, L_2 (Q)) \leq (A^*_1 / \ve)^{\nu^*_1}$ and $\sup_Q N (\ve U^*_2, \mathcal F_1 + \mathcal F_2, L_2 (Q)) \leq (A^*_2 / \ve)^{\nu^*_2}$, for any $0 < \ve < 1$ where $\mathcal F_1 \mathcal F_2 := \{f_1 f_2 : f_1 \in \mathcal F_1, f_2 \in \mathcal F_2 \}$
	and $\mathcal F_1 + \mathcal F_2 := \{f_1 + f_2 : f_1 \in \mathcal F_1, f_2 \in \mathcal F_2 \}$. \\
	%%(b) $\{Z'\beta : \beta \in \Theta^{\beta} \}$, $\{ U'\gamma : \gamma \in \Theta^{\gamma} \}$, $\{ (V + X'\psi) / h_n : \psi \in \Theta^{\psi} \}$ and $\{ \mathcal K ((V + X'\psi)/h) : \psi \in \Theta^\psi, h \in (0, \infty) \}$ are VC classes with VC dimension at most $p_1 + 1$, $p_2 + 2$, $q + 3$ and $q + 3$ respectively. \\
	%%(c) There exist $0 < U, A, \nu < \infty$ independent of $n$ such that $\sup_{Q} N (\ve U, \{ \mathcal K'((V + X'\psi)/h_n) : \psi \in \Theta^{\psi} \}, L_2(Q)) \leq (A/\ve)^{\nu}$ holds for any $0 < \ve < 1$. \\
	(e) Let $\phi:\mathbb R \rightarrow \mathbb R$ be a Lipschitz continuous function and $\mathcal F$ be a class of function with a constant envelope $F$ such that there exist $0 < U, A, \nu < \infty$
	satisfying $\sup_Q N (\ve U, \mathcal F, L_2 (Q)) < (A / \ve)^{\nu}$ for any $0 < \ve < 1$. Then there exist $0 < U^*, A^*, \nu^* < \infty$ such that $\sup_Q N(\ve U^*, \phi (\mathcal F), L_2 (Q)) \leq (A^* / \ve)^{\nu^*}$
	for any $0 < \ve < 1$ where $\phi (\mathcal F) := \{ \phi(f) : f \in \mathcal F \}$. \\
	(f) Let $\Theta \subset \mathbb R^d$ be a non-empty bounded subset with diameter $D$, and let $\mathcal F = \{f_{\theta} : \theta \in \Theta \}$ be a class of functions defined on $S$ indexed by $\Theta$
	such that, for any $\theta_1$ and $\theta_2$, $\sup_{x \in S} | f_{\theta_1} (x) - f_{\theta_2} (x) | \leq M \| \theta_1 - \theta_2 \|$ for $0 < M < \infty$. Then there exist $0 < U, A, \nu < \infty$ such that
	$\sup_{Q} N (\ve U, \mathcal F, L_2 (Q)) \leq (A/\ve)^{\nu}$ for any $0 < \ve < 1$.
\end{lemma}

\begin{proof}[Proof]
	(a) This is a direct consequence of Theorem 9.3 of \cite{kosorok2008introduction}. (b) This is Lemma 9.6 of \cite{kosorok2008introduction}. (c) This is Lemma 9.9(viii) of \cite{kosorok2008introduction}. (d) This is a direct consequence of Corollary 7(i) of \cite{katoep}.
	(e) This is a direct consequence of Proposition 5 of \cite{katoep}. (f) This is a direct consequence of Lemma 26 of \cite{katoep}. 
\end{proof}

\begin{lemma}\label{lem_roc}
Suppose that the assumption of Proposition \ref{roc psi} holds and $\Theta_n$ is defined, as in the proof of Proposition \ref{roc psi}. Then, \\
(a) $\sup_{\theta \in \Theta} |  P_n \Psi_n^{(1)} (\theta) - P \Psi_n^{(1)} (\theta)  | = o_p(1)$, \\
(b) $\sup_{\theta \in \Theta_n} \left|P \Psi_n^{(1)} (\theta) - 
	\int_{0}^{\tau} \mathbb E[Y(t) (U'\gamma_0) (X / h_n) \mc K' ((V + X' \psi)/h_n) e^{\eta (W, \theta_0)}] d\Lambda_0 (t) \right| = o(1)$, \\
(c) $\sup_{\theta \in \Theta} \left| P_n \left( \delta \frac{\widetilde P_n \Psi^{(3)}_n (\theta, T)}{\widetilde P_n \Psi^{(2)}_n (\theta, T)} \right) - 
	P_n  \left( \delta \frac{P \Psi^{(3)}_n (\theta, T)}{P \Psi^{(2)}_n (\theta, T)} \right) \right| = o_p (1)$, \\
(d) $\sup_{\theta \in \Theta_n} \left| P_n \left( \delta \frac{P \Psi^{(3)}_n (\theta, T)}{P \Psi^{(2)}_n (\theta, T)} \right)- P_n \left(\delta \frac{P \Psi^{(3)}_n (\theta, T)}{P \Psi^{(4)} (\theta_0, T)} \right) \right| = o_p (1)$, \\
(e) $\sup_{\theta \in \Theta} \left|P_n \left(\delta \frac{P \Psi^{(3)}_n (\theta, T)}{P \Psi^{(4)} (\theta_0, T)} \right) -
	P \left(\delta \frac{P \Psi^{(3)}_n (\theta, T)}{P \Psi^{(4)} (\theta_0, T)} \right) \right| = o_p (1)$, \\
(f) $\sup_{\theta \in \Theta_n} \left|P \left(\delta \frac{P \Psi^{(3)}_n (\theta, T)}{P \Psi^{(4)} (\theta_0; T)} \right) -
\int_{0}^{\tau} \mathbb E[Y (t) U' \gamma_0 (X / h_n) \mc K'((V + X' \psi)/h_n) e^{Z' \beta_0 + U' \gamma_0 \mc K ((V + X ' \psi)/h_n)}] d \Lambda_0 (t) \right| = o (1).$
\end{lemma}

\begin{proof}
In the following, without loss of generality, we occasionally treat $X$ as a scalar, for brevity.

\underline{Part (a).}
We prove part(a) by applying inequality (2.5) in \cite{gine2001consistency} to $\{ h_n \Psi_n^{(1)} (\theta) : \theta \in \Theta  \}$. To this end, we  verify that, for a constant envelope $U$ and $0 < \ve < 1$, there exist $A$ and $\nu$ such that $\sup_{Q} N (\ve U, \{ h_n \Psi_n^{(1)} (\theta) : \theta \in \Theta  \}, L_2 (Q)) \leq (A/\ve)^{\nu}$. 
By Lemma \ref{empirical process}(d), it suffices to show that there exist $U^{\ast}$, $A^{\ast}$ and $\nu^{\ast}$ such that $\sup_{Q} N (\ve U^{\ast}, \mc H, L_2 (Q)) \leq (A^{\ast} /\ve)^{\nu^{\ast}}$ for $0 < \ve < 1$, with $\mc H = \{\delta \}, \{X\}, \{U'\gamma : \gamma \in \Theta^{\gamma} \}$ and $\{\mc K' ((V + X' \psi)/h_n) : \psi \in \Theta^{\psi} \}$, respectively. The inequality trivially holds for $\{ \delta \}$ and $\{ X \}$. 
Further, the inequality holds for $\{U'\gamma : \gamma \in \Theta^{\gamma} \}$, by Lemma \ref{empirical process}(a) and (b). 
For $\{\mc K' ((V + X ' \psi)/h_n) : \psi \in \Theta^{\psi} \}$, observe that $\{\mc (V + X ' \psi)/h_n : \psi \in \Theta^{\psi} \}$ is a VC-class with a VC-dimension at most $q + 3$, by Lemma \ref{empirical process}(b). By Assumption \ref{kernel and smoothing 2}(a), it follows from Lemma 3.6.11 and Exercise 3.6.14 of \cite{gine2021mathematical} that there exists a nondecreasing function $h: \mbb R \rightarrow \mbb [0, \nu_1 (\mc K')]$ and a Lipschitz continuous function $g: \mbb R \rightarrow \mbb R$ with $\sup_{u \in \mbb R} |g (u)| \leq \sup_{u \in \mbb R} |\mc K' (u)| < \infty$ such that $\mc K' = g \circ h$ 
where $\nu_1 (\mathcal K') := \sup \{ \sum_{i = 1}^n |\mathcal K' (x_i) - \mathcal K'(x_{i - 1})| : - \infty < x_0 < \dots < x_n < \infty, n \in \mathbb N \}$. 
Then, by Lemma \ref{empirical process}(c), $\{ h((V + X' \psi)/ h_n) : \psi \in \Theta^{\psi} \}$ is a VC-class with a VC-dimension at most $q + 3$ and with a constant envelope $\nu_1 (\mc K')$. Thus, there exist $A^{\circ}$ and $\nu^{\circ}$ independent of $n$ such that $\sup_Q N ( \ve \nu_1 (\mc K'), \{ h((V + X' \psi)/ h_n) : \psi \in \Theta^{\psi} \}, L_2 (Q)) \leq (A^{\circ} / \ve)^{\nu^{\circ}}$ holds for $0 < \ve < 1$ by Lemma \ref{empirical process}(a). 
Since $g$ is Lipschitz continuous and $\sup_{u \in \mbb R} |g (u)| \leq \sup_{u \in \mbb R} |\mc K' (u)|$, it follows from Lemma \ref{empirical process}(e) that
$\sup_Q N (\ve U^{\ast}, \{\mc K' ((V + X ' \psi)/h_n) : \psi \in \Theta^{\psi} \}, L_2 (Q)) \leq (A^{\ast} /\ve)^{\nu^{\ast}}$ holds for some $U^{\ast}, A^{\ast}$ and $\nu^{\ast}$, independent of $n$ where $0 < \ve < 1$. Consequently, we obtain the desired bound:  $\sup_{Q} N (\ve U, \{ h_n \Psi_n^{(1)} (\theta) : \theta \in \Theta  \}, L_2 (Q)) \leq (A/\ve)^{\nu}$.

By construction, we can choose $U$, $A$, and $\nu$ in the previous paragraph such that they do not depend on $n$. Therefore, the application of  inequality (2.5) in \cite{gine2001consistency} yields
\begin{equation} \label{ep_bound_2}
	\mathbb E \sup_{\theta \in \Theta} \left|  P_n \Psi_n^{(1)} (\theta) - P \Psi_n^{(1)} (\theta)  \right| \leq \mc C \left[ \frac{\nu U}{n h_n} \log \left( \frac{AU}{\sigma_n} \right) + \frac{\sqrt{\nu} \sigma_n}{\sqrt{n} h_n}
	\sqrt{\log \left( \frac{AU}{\sigma_n} \right)} \right],
\end{equation}
for every $n \in \mbb N$ with $\sigma_n^2 := \sup_{\theta \in \Theta} P (h_n \Psi_n^{(1)} (\theta))^2$. We proceed to show that $\sigma_n^2 \geq \mc C h_n$ for sufficiently large $n$. A straightforward calculation with the change of variable yields
\begin{align}
\sigma^2_n &\geq P (h_n \Psi_n^{(1)} (\theta_0))^2 \notag \\
& = h_n \int_{0}^{\tau} d \Lambda_0 (t) \int d F_R \int 1 \{t' \geq t \} x^2 (u'\gamma_0)^2 \mathcal K' (s)^2 e^{z'\beta_0 + u'\gamma_0 1 \{s h_n \geq 0 \} } f_{S|R} (sh_n | r) ds \notag \\
& \geq \mathcal C h_n \int_{0}^{\tau} 1 \{t' \geq t \} x^2 (u'\gamma_0)^2 \int_{- \infty}^{\infty} \mathcal K' (s)^2 (f_{S|R} (0 | r) + |s|O(h_n)) dsdF_R d\Lambda_0(t) \notag \\
& \geq \mathcal C h_n \int_{0}^{\tau} \mathbb E[Y(t) X^2 (U'\gamma_0)^2 f_{S|R} (0|R)] d\Lambda_0 (t) \left( \int_{- \infty}^{\infty} \mathcal K'(s)^2 ds \right) + O(h_n^2), \notag
\end{align}
where we use the mean value theorem for the second inequality and $\int |s| \mathcal K'(s)^2 ds < \infty$ for the third inequality.
It now follows that, for sufficiently large $n$, $\sigma_n^2 \geq \mc C h_n$. Meanwhile, it follows from BA and Assumption \ref{kernel and smoothing 2}(b) that $\sigma_n^2 \leq \mc C h_n$. Therefore, by Assumption \ref{kernel and smoothing 2}(c), the right-hand side of (\ref{ep_bound_2}) converges to zero, and part (a) follows.

\underline{Part (b).}
Observe that
\begin{align}\label{step1b_1}
	&\sup_{\theta \in \Theta_n} \left|P \Psi_n^{(1)} (\theta) - 
	\int_{0}^{\tau} \mathbb E[Y(t) (U'\gamma_0) (X / h_n) \mc K' ((V + X' \psi)/h_n) e^{\eta_0 (W, \theta_0)}] d\Lambda_0 (t) \right| \notag \\
  = & \sup_{\theta \in \Theta_n} \left| \int \mathbb E[Y(t) U'(\gamma - \gamma_0) (X/h_n) \mc K' ((V + X' \psi)/h_n) e^{\eta_0 (W, \theta_0)}] d\Lambda_0 (t) \right| \notag \\
	\leq & r_n \mc C \sup_{\theta \in \Theta_n} \mathbb E[(1/h_n)  \mc K' ((V + X' \psi)/h_n)],
\end{align}
where the inequality follows from the Cauchy-Schwarz inequality (CS), BA, and the definition of $\Theta_n$.

Here, we note that, by the change of variable,
\begin{align}\label{kernel inequality}
	\sup_{\theta \in \Theta_n} \mathbb E[(1/h_n)  \mc K' ((V + X' \psi)/h_n)] = \sup_{\theta \in \Theta_n} \int dF_{R} \int \mc K' (s) f_{S | R} (s h_n - x' (\psi - \psi_0)) ds < \infty,
\end{align}
where the last inequality follows from Assumption \ref{kernel and smoothing 2}(b) and \ref{density}. Hence, the right-hand side of (\ref{step1b_1}) converges to zero. This completes the proof of part (b).

\underline{Part (c).}
Observe that
\begin{align}\label{step1c_1}
	&\sup_{\theta \in \Theta} \left| P_n \left( \delta \frac{\widetilde P_n \Psi^{(3)}_n (\theta, T)}{\widetilde P_n \Psi^{(2)}_n (\theta, T)} \right) - 
	P_n  \left( \delta \frac{P \Psi^{(3)}_n (\theta, T)}{P \Psi^{(2)}_n (\theta, T)} \right) \right| \notag \\
	\leq &\frac{\mc C}{P_n (Y(\tau))\mathbb E[Y(\tau)]} \sup_{\theta \in \Theta, t \in [0, \tau]} 
 \left|P \Psi^{(2)}_n (\theta, t)P_n \Psi^{(3)}_n (\theta, t) -  P_n \Psi^{(2)}_n (\theta, t) P \Psi^{(3)}_n (\theta, t)\right| \notag \\
	\leq &\frac{\mc C}{P_n (Y(\tau)) \mathbb E[Y(\tau)]} \left[ \left(\sup_{\theta \in \Theta, t \in [0, \tau]} P_n \Psi^{(3)}_n (\theta, t)\right) \mc A_n^{(1)}
 +  \left(\sup_{\theta \in \Theta, t \in [0, \tau]} P_n \Psi^{(2)}_n (\theta, t)\right) \mc A_n^{(2)} \right],
\end{align}
where $\mc A_n^{(1)} := \sup_{\theta \in \Theta, t \in [0, \tau]} |P_n \Psi^{(2)}_n (\theta, t) - P \Psi^{(2)}_n (\theta, t) |$ and $\mc A_n^{(2)} : = \sup_{\theta \in \Theta, t \in [0, \tau]} |P_n \Psi^{(3)}_n (\theta, t) - P \Psi^{(3)}_n (\theta, t) |$. Note that $\sup_{\theta \in \Theta, t \in [0, \tau]} P_n \Psi^{(2)}_n (\theta, t)$ is bounded and $\sup_{\theta \in \Theta, t \in [0, \tau]} |P_n \Psi^{(3)}_n (\theta, t)| \leq \mc A_n^{(2)} + \sup_{\theta \in \Theta, t \in [0, \tau]} |P \Psi^{(3)}_n (\theta, t)| $ where 
$\sup_{\theta \in \Theta, t \in [0, \tau]} |P \Psi^{(3)}_n (\theta, t)|$ is uniformly bounded in $n$ by the proof of part (b). Moreover, $P_n Y(\tau) \rightarrow_p \mathbb E[Y(\tau)] > 0$ by the law of large numbers and Assumption \ref{a survival time and a censoring time}(b). Hence, it suffices to show that $\mc A_n^{(1)}$ and $\mc A_n^{(2)}$ converge
to zero in probability for the stated result.% where we use an inequality $|a_1 b_1 - a_2 b_2| \leq |b_1| |a_1 - a_2| + |a_2| |b_1 - b_2|$ for $a_1, a_2, b_1, b_2 \in \mbb R$.

First, for the convergence of $\mc A_n^{(1)}$, it suffices to show that $ \mathcal A^{*} := \{Y(t) e^{Z'\beta + U'\gamma \mathcal K ((V + X'\psi)/h)} : \theta \in \Theta, h \in (0, \infty), t \in [0, \tau] \}$ is Glivenko-Cantelli. Recall that $\mc K$ is monotone by Assumption \ref{kernel and smoothing}(a). 
Thus, $\{ \mc K((V + X' \psi)/h) : \psi \in \Theta^{\psi}, h \in (0, \infty) \}$ is a VC class, by Lemma \ref{empirical process}(b) and (c). 
It follows from Lemma \ref{empirical process}(a), (b), and (d) that there exists $0 < U, A, \nu < \infty$ such that $\sup_Q N(\ve U, \{Z'\beta + U'\gamma \mathcal K((V + X'\psi)/h) : \theta \in \Theta, h \in (0, \infty) \}, L_2 (Q)) \leq (A / \ve)^{\nu}$.
It must be noted that $Z'\beta + U'\gamma \mathcal K((V + X'\psi)/h)$ is bounded by BA. 
Since the exponential function is Lipschitz continuous on a bounded interval and $\{Y(t) : t \in [0, \tau] \}$ is VC-class, there exists $0 < U, A, \nu < \infty$ such that $\sup_Q N (\ve U, \mathcal A^*, L_2 (Q) ) \leq (A / \ve)^{\nu}$ for $0 < \ve < 1$ by Lemma \ref{empirical process}(a), (d), and (e).
Hence, $\mathcal A^{(1)}_n \rightarrow_p 0$.

For $\mc A_n^{(2)}$, it follows from a similar argument as that in the previous paragraph and that in the proof of part (a), in conjunction with Lemma \ref{empirical process}(a), (b), and (d), that 
there exist a constant envelope $U^{\circ}$, $A^{\circ}$, and $\nu^{\circ}$, independent of $n$ such that $\sup_Q N (\ve U^{\circ}, \{ h_n \Psi^{(3)}_n (\theta, t) : \theta \in \Theta, t \in [0, \tau] \}, L_2(Q)) \leq (A^{\circ} / \ve)^{\nu^{\circ}}$ for $0 < \ve < 1$. Therefore, applying inequality (2.5) in \cite{gine2001consistency}, similar to the argument in part (a), yields $\mc A_n^{(2)} \rightarrow_p 0$. This completes part (c). 

\underline{Part (d).}
Observe that
\begin{align}\label{step1d_1}
	\sup_{\theta \in \Theta_n} \left| P_n \left( \delta \frac{P \Psi^{(3)}_n (\theta, T)}{P \Psi^{(2)}_n (\theta, T)} \right)- P_n \left(\delta \frac{P \Psi^{(3)}_n (\theta, T)}{P \Psi^{(4)} (\theta_0, T)} \right) \right|
	\leq \mc C \sup_{\theta \in \Theta_n} P_n  |P \Psi^{(2)}_n (\theta, T) - P \Psi^{(4)} (\theta_0, T)|,
\end{align}
by BA and the argument in part (b). It follows from BA and CS that the right-hand side of (\ref{step1d_1}) is bounded by $\mc C \sup_{\theta \in \Theta_n} \left( \| \beta - \beta_0 \| \mathbb E[\|Z\|] + \| \gamma - \gamma_0 \| \mathbb E[\|U\|] + \mathbb E[|\mc K((V + X' \psi)/h_n) - 1 \{ V + X' \psi_0  \geq 0\}|]\right) = \mc C \sup_{\theta \in \Theta_n} \mathbb E[|\mc K((V + X' \psi)/h_n) - 1 \{ V + X' \psi_0  \geq 0\}|] + o(1)$. Fix $\alpha > 0$. Then, observe that
\begin{align}\label{step1d_2}
	&\sup_{\theta \in \Theta_n}\mathbb E[|\mc K ((V + X' \psi)/ h_n) - 1\{ V + X' \psi_0 \geq 0 \}|] \notag \\
	\leq &\sup_{\theta \in \Theta_n}\mathbb E[|\mc K ((V + X' \psi)/ h_n) - 1\{ V + X' \psi_0 \geq 0 \}| 1 \{ |V + X' \psi| \geq \alpha\}] + \mc C \sup_{\theta \in \Theta} \mbb P (|V + X' \psi| < \alpha ). %\notag \\
	%\leq & \max \{ \mc K(- \alpha / h_n), |1 - \mc K (\alpha / h_n)| \} + \mc C \sup_{\theta \in \Theta} \mbb P (|V + X' \psi| < \alpha )
\end{align}
The first term on the right-hand side of (\ref{step1d_2}) equals $\sup_{\theta \in \Theta_n}\mathbb E[|\mc K ((V + X' \psi)/ h_n) - 1\{ V + X' \psi \geq 0 \}| 1 \{ |V + X' \psi| \geq \alpha\}]$ for sufficiently large $n$ by the construction of $\Theta_n$. It must be noted that $\sup_{\theta \in \Theta_n}\mathbb E[|\mc K ((V + X' \psi)/ h_n) - 1\{ V + X' \psi \geq 0 \}| 1 \{ |V + X' \psi| \geq \alpha\}] \leq \max \{ \mc K(- \alpha / h_n), |1 - \mc K (\alpha / h_n)| \}$, which converges to zero by Assumption \ref{kernel and smoothing}(a). Since $\lim_{\alpha \downarrow 0} \sup_{\theta \in \Theta} \mbb P (|V + X' \psi| < \alpha ) = 0$, by the proof of Lemma 4 of \cite{horowitz1992smoothed}, $\sup_{\theta \in \Theta_n} \mathbb E[|\mc K((V + X' \psi)/h_n) - 1 \{ V + X' \psi_0  \geq 0\}|] = o(1)$ follows. This completes the proof of part (d).

\underline{Part (e).}
For sufficiently large $n$, $h_n \in (0, 1)$. Thus, it suffices to show that
\begin{equation}\label{part_e_1}
	\sup_{\theta \in \Theta, h \in (0, 1)} \left|P_n \left(\delta \frac{P \widetilde \Psi^{(3)} (\theta, h,  T)}{P \Psi^{(4)} (\theta_0, T)} \right) -
	P \left(\delta \frac{P \widetilde \Psi^{(3)} (\theta, h,  T)}{P \Psi^{(4)} (\theta_0, T)} \right) \right| = o_p(1),
\end{equation}
where $\widetilde \Psi^{(3)} (\theta, h, t) := \{ Y(t) (U'\gamma) (X/h) \mathcal K' ((V + X'\psi)/h) e^{Z'\beta + U'\gamma \mathcal K((V + X'\psi)/h)} : \theta \in \Theta, h \in (0, 1) \}$.
Take $\theta_1, \theta_2 \in \Theta$, and $h_1, h_2 \in (0, 1)$ arbitrarily. Then, for any realization $t \in [0, \tau]$ of $T$,
\begin{align}\label{part_e_2}
	\left|\delta \frac{P \widetilde \Psi^{(3)} (\theta_1, h_1, t)}{P \Psi^{(4)} (\theta_0, t)} - \delta \frac{P \widetilde \Psi^{(3)} (\theta_2, h_2, t)}{P \Psi^{(4)} (\theta_0, t)} \right| \leq \mc C  |P \widetilde \Psi^{(3)} (\theta_1, h_1, t) -  P \widetilde \Psi^{(3)} (\theta_2, h_2, t) | \leq \mc C (\mc A^{(1)} + \mc A^{(2)}),
\end{align}
where
\begin{align*}
	\mc A^{(1)} &:= \left|\mathbb E\left[Y(t) U' (\gamma_1 - \gamma_2) (X / h_1) \mc K' ((V + X' \psi_1)/h_1) e^{Z' \beta_1 + U'\gamma_1 \mc K((V + X' \psi_1)/h_1)}\right]\right|, \\
	\mc A^{(2)} &:= \left|\mathbb E \left[Y(t) U' \gamma_2 (X / h_1) \mc K' ((V + X' \psi_1)/h_1) e^{Z' \beta_1 + U'\gamma_1 \mc K((V + X' \psi_1)/h_1)} \right] \right. \\ &\left. -
	\mathbb E\left[Y(t) U' \gamma_2 (X / u_2) \mc K' ((V + X' \psi_2)/h_2) e^{Z' \beta_2 + U'\gamma_2 \mc K((V + X' \psi_2)/h_2)}\right] \right|.
\end{align*}
It follows from the argument in part (b), in conjunction with Assumption \ref{covariates assumption}(a) and CS, that $\mc A^{(1)}$ is bounded by $\mc C \| \gamma_1 - \gamma_2\|$.

Meanwhile, by Lipschitz continuity of the exponential function and Assumptions \ref{covariates assumption}(a) and \ref{density}, the change in variable gives that $\mc A^{(2)}$ is bounded by
\begin{align}\label{part_e_3}
	&\mc C \int dF_{R} \int |\mc K'(s)| \left| e^{z\beta_1 + u'\gamma_1 \mc K(s)} f_{S | R} (s h_1 - x' (\psi_1 - \psi_0)| r) - e^{z \beta_2 + u'\gamma_2 \mc K(s)} f_{S | R} (s h_2 - x' (\psi_2 - \psi_0)| r)\right| ds \notag \\
	\leq & \mc C \left[ \int dF_{R} \int |\mc K'(s)| \cdot |z'(\beta_1 - \beta_2) + \mc K(s) u'(\gamma_1 - \gamma_2)|ds \right. \notag \\
	& \left. + \int dF_{R} \int |\mc K'(s)| \cdot |f_{S | R} (s h_1 - x' (\psi_1 - \psi_0)| r) -  f_{S | R} (s h_2 - x' (\psi_2 - \psi_0)| r)|ds \right].
\end{align}
By CS and the mean value theorem, in conjunction with Assumption \ref{density}, the last term of (\ref{part_e_3}) is bounded by $\mc C \left[ \| \beta_1 - \beta_2 \| + \| \gamma_1 - \gamma_2\| + \int dF_{R} \int | \mc K'(s) | \left\{|s| \cdot |h_1 - h_2| + \| \psi_1 - \psi_2 \| \right\}ds \right]$. Hence, it follows from Assumption \ref{kernel and smoothing 2}(b) that $\mc A^{(2)}$ is bounded by $ \mc C \left\|(\theta_1', h_1) - (\theta_2', h_2)\right\|$ where $\mc C$ does not depend on the value of $t$.

Based on the above argument on $\mc A^{(1)}$ and $\mc A^{(2)}$, the last term in (\ref{part_e_2}) is bounded by $M \| (\theta_1', h_1) - (\theta_2', h_2) \|$ for some $M > 0$, irrespective of the value of $t$. 
Since $\Theta$ is compact, Lemma \ref{empirical process}(f) implies that $\{ \delta P \widetilde \Psi^{(3)} (\theta, h, T)/ P \Psi^{(4)} (\theta_0, T): \theta \in \Theta, h \in (0, 1) \}$ is Glivenko-Cantelli. Thus, (\ref{part_e_1}) holds and this completes the proof of part (e). 

\underline{Part (f).}
Observe that $P \left(\delta P \Psi^{(3)}_n (\theta, T) / P \Psi^{(4)} (\theta_0, T) \right) = \mathbb E\left[\int_{0}^{\tau}  \frac{P \Psi^{(3)}_n (\theta, t)}{P \Psi^{(4)} (\theta_0, t)} Y(t)e^{\eta_0 (W, \theta_0)}d\Lambda_0 (t) \right] = \int_{0}^{\tau} P \Psi^{(3)}_n (\theta, t) d\Lambda_0 (t)$ where we use Lemma \ref{martingale property} and Fubini's Theorem for the first and equalities, respectively. Hence, we obtain the following bound:
\begin{align}\label{part_f_1}
	&\sup_{\theta \in \Theta_n} \left|P \left(\delta \frac{P \Psi^{(3)}_n (\theta, T)}{P \Psi^{(4)} (\theta_0, T)} \right) -
\int_{0}^{\tau} \mathbb E[Y (t) U' \gamma_0 (X / h_n) \mc K' \left( \frac{V + X' \psi}{h_n} \right) e^{Z' \beta_0 + U' \gamma_0 \mc K \left( \frac{V + X' \psi}{h_n} \right)}] d \Lambda_0 (t) \right| \notag \\
\leq & \sup_{\theta \in \Theta_n} \left| \int_{0}^{\tau} \mathbb E [Y(t) U'(\gamma - \gamma_0) (X / h_n) \mc K'((V + X' \psi)/h_n) e^{Z' \beta + U' \gamma \mc K ((V + X ' \psi)/h_n)}] d\Lambda_0 (t) \right| \notag \\
+ &\sup_{\theta \in \Theta_n} \left| \int_{0}^{\tau} \mathbb E \left[Y(t) U'\gamma_0 (X / h_n) \mc K' \left( \frac{V + X' \psi}{h_n} \right) \left(e^{Z' \beta + U' \gamma \mc K \left( \frac{V + X' \psi}{h_n} \right)} -  e^{Z' \beta_0 + U' \gamma_0 \mc K \left( \frac{V + X' \psi}{h_n} \right)}\right) \right] d\Lambda_0 (t) \right|.
\end{align}
Based on the argument in part (b), in conjunction with CS and Assumptions \ref{the parameter space}(a) and \ref{covariates assumption}(a), the first term on the right-hand side of (\ref{part_f_1}) is bounded by $\mc C r_n$, which is convergent to zero. By a similar argument, the second term converges to zero. This completes the proof of part (f).
\end{proof} 

\begin{lemma}\label{lem: matrix}
Let $X$ be a $\mbb R^{k}$-valued bounded random variable, with $\mathbb E[X X']$ being positive definite. 
Then, for any vector $u$ in $\mathbb R^k$ such that $\| u \| = 1$, we have $P (|X'u| > c_1) > c_2$ for some positive constants $c_1$ and $c_2$ that do not depend on $u$.
\end{lemma}

\begin{proof}[Proof]
Let $\lambda$ be the smallest eigenvalue of $\mathbb E[X X']$. By the positive definiteness of $\mathbb E[X X']$, $\mathbb E[|X' u|^2] = u' \mathbb E[X X'] u \geq \lambda > 0$, for any $u \in \mbb R^k$ such that $\| u \| = 1$. It follows from the Paley-Zygmund inequality (e.g., equation (11) of \cite{petrov2007lower}) that $P(|X'u|^2 > a \mathbb E[|X'u|^2]) \geq (1 - a)^2 \left(\mathbb E[|X'u|^2]\right)^2 / \mathbb E[|X'u|^4]$ for any $0 \leq a \leq 1$. By the boundedness of $X$, there exists a finite positive constant $C$ such that $\mathbb E[|X'u|^4] < C$ uniformly in $u$. Therefore, setting $a = 0.5$ so that $c_1 = (0.5)^{1/2} \lambda^{1/2}$ and $c_2 = (0.5)^2 \lambda^{2}/C$ gives the desired result.
\end{proof}

\begin{lemma}\label{error indicator kernel}
	Suppose that the assumptions of Proposition \ref{a score} hold. Then,
	\[ 
		n^{- 1/2} \sum_{i = 1}^n |1 \{V_i + X_i' \psi_0 \geq 0 \} - \mc K ((V_i + X_i' \psi_0)/ h_n)| = o_p (1).
	\]
\end{lemma}

\begin{proof}[Proof]
	Observe that
\begin{align}\label{a score 12}
	&\mathbb E \left[ n^{- 1/2} \sum_{i = 1}^n |1 \{V_i + X_i' \psi_0 \geq 0 \} - \mc K ((V_i + X_i' \psi_0)/ h_n)| \right] \notag \\
	= &n^{1/2} \int d F_R \int |1 \{ s \geq 0 \} - \mc K (s/h_n)| f_{S | R} (s | r) ds \notag \\
	= &n^{1/2} h_n  \int d F_R \int |1 \{ s h_n \geq 0 \} - \mc K (s)| f_{S | R} (s h_n | r) ds \notag \\
	\leq &\mc C \sqrt{n} h_n  \int d F_R \int |1 \{ s \geq 0 \} - \mc K (s)| ds,
\end{align}
where the second equality follows from the change in variable and the inequality from Assumption \ref{density}. Here, it must be noted that integration by parts gives
\begin{align*}
	\int |1 \{ s \geq 0 \} - \mc K (s)| ds &= \int_{ - \infty}^{0} \mc K (s) ds + \int_{0}^{\infty} (1 - \mc K (s)) ds \\
	&= \lim_{R \rightarrow - \infty } \left[ s \mc K (s) \right]_{R}^{0} + \lim_{R \rightarrow \infty} \left[s (1 - \mc K (s)) \right]_{0}^{R} + \int_{- \infty}^{\infty} |s| \mc K' (s) ds.
\end{align*}
By L'H\^{o}pital's rule and Assumption \ref{kernel and smoothing 2}(b), $\lim_{R \rightarrow - \infty} R \mc K (R) = \lim_{R \rightarrow - \infty} - R^2 \mc K' (R) = 0$. By a similar argument, we have $\lim_{R \rightarrow \infty}  R (1 - \mc K (R)) = 0$. Hence, by Assumption \ref{kernel and smoothing 2}(b), $\int |1 \{ s \geq 0 \} - \mc K (s)| ds < \infty$. 
It now follows from (\ref{a score 12}) and Assumption \ref{kernel and smoothing 2}(c) that the right-hand side of (\ref{a score 12}) converges to zero. Since $L_1$ convergence implies convergence in probability, we complete the proof.
\end{proof}

\begin{lemma}\label{error two indicator}
	Suppose that the assumptions of Proposition \ref{a score} hold. Then, for any consistent estimator $\widetilde \psi_n$ of $\psi_0$,
	\[ 
		\frac{1}{n} \sum_{i = 1}^n |1 \{ V_i + X_i'\widetilde \psi_n \geq 0 \} - 1 \{V_i + X_i'\psi_0 \geq 0 \}| = o_p(1).
	\]
\end{lemma}

\begin{proof}[Proof]
	Since $\widetilde \psi_n \rightarrow_p \psi_0$, there exists a sequence $r_n \ (n \in \mathbb N)$ such that $r_n \rightarrow 0$ and $\mathbb P (\| \widetilde \psi_n - \psi_0 \| \geq r_n) \rightarrow 0$ as $n \rightarrow \infty$. 
Let $\Theta_n := \{ \psi : \| \psi - \psi_0 \| < r_n \}$. Observe that the left-hand side is bounded by $\sup_{\theta \in \Theta_n} n^{-1} \sum_{i = 1}^n |1 \{ V_i + X_i'\psi \geq 0 \} - 1 \{V_i + X_i'\psi_0 \geq 0 \}|$, with probability approaching one. Further,
\begin{align}\label{error ind}
	&\mathbb E \left[ \sup_{\theta \in \Theta_n} n^{-1} \sum_{i = 1}^n |1 \{ V_i + X_i'\psi \geq 0 \} - 1 \{V_i + X_i'\psi_0 \geq 0 \}| \right ] \notag \\
	\leq &\mathbb E \left[ \sup_{\theta \in \Theta_n} | 1 \{ V_i + X_i'\psi \geq 0 \} - 1 \{V_i + X_i'\psi_0 \geq 0 \} | \right].
\end{align}
It must be noted that, for fixed $\alpha > 0$,
\begin{align}\label{error ind expectation} 
	&E \left[ \sup_{\theta \in \Theta_n} |1 \{ V + X' \psi \geq 0 \} - 1 \{V + X'\psi_0 \geq 0 \}| \right] \notag \\
	\leq &E \left[ \sup_{\theta \in \Theta_n} |1 \{ V + X' \psi \geq 0 \} - 1 \{V + X'\psi_0 \geq 0 \}| 1\{|V + X'\psi_0| \geq \alpha \} \right] + 2 \mathbb P (|V + X'\psi_0| < \alpha).
\end{align}
For sufficiently large $n$, the first term on the right-hand side of (\ref{error ind expectation}) is zero, while the second term converges to zero, as $\alpha \rightarrow 0$ by the proof of lemma 4 of \cite{horowitz1992smoothed}.
Thus, the right-hand side of (\ref{error ind}) is $o(1)$. Since $L_1$ convergence implies the convergence in probability, we complete the proof.
\end{proof}

\bibliographystyle{asa}
\bibliography{cpcox}

\clearpage

%\begin{thebibliography}{}
%	\item[]Seo, M. H., Linton, O. (2007). A smoothed least squares estimator for threshold regression models. Journal of Econometrics, 141(2), 704-735.
%\end{thebibliography}

\end{document}